\documentclass[10pt,oneside]{amsart}

\usepackage[
paper=a4paper,
text={137mm,205mm},centering
]{geometry}

\usepackage{amssymb,amsxtra}
\usepackage{float, placeins}
\usepackage[all]{xy}
\usepackage{pb-diagram,pb-xy}
\usepackage{graphicx,color,float}
\usepackage[bookmarks]{hyperref}
\usepackage{pinlabel}
\usepackage[
  textwidth=1.2in,
  backgroundcolor=yellow,
  bordercolor=orange,
  textsize=footnotesize
]{todonotes}

\iftrue
\makeatletter
\def\@settitle{%
  \vspace*{-20pt}
  \begin{flushleft}%
    \baselineskip14\p@\relax
    \normalfont\bfseries\LARGE
    \@title
  \end{flushleft}%
}
\def\@setauthors{%
  \begingroup
  \def\thanks{\protect\thanks@warning}%
  \trivlist
  \raggedright
  \large \@topsep30\p@\relax
  \advance\@topsep by -\baselineskip
  \item\relax
  \author@andify\authors
  \def\\{\protect\linebreak}%
  \authors
  \ifx\@empty\contribs
  \else
    ,\penalty-3 \space \@setcontribs
    \@closetoccontribs
  \fi
  \normalfont
  \@setaddresses
  \endtrivlist
  \endgroup
}
\def\@setaddresses{\par
  \nobreak \begingroup
  \small
  \def\author##1{\nobreak\addvspace\smallskipamount}%
  \def\\{\unskip, \ignorespaces}%
  \interlinepenalty\@M
  \def\address##1##2{\begingroup
    \par\addvspace\bigskipamount\noindent
    \@ifnotempty{##1}{(\ignorespaces##1\unskip) }%
    {\ignorespaces##2}\par\endgroup}%
  \def\curraddr##1##2{\begingroup
    \@ifnotempty{##2}{\nobreak\noindent\curraddrname
      \@ifnotempty{##1}{, \ignorespaces##1\unskip}\/:\space
      ##2\par}\endgroup}%
  \def\email##1##2{\begingroup
    \@ifnotempty{##2}{\nobreak\noindent E-mail address%
      \@ifnotempty{##1}{, \ignorespaces##1\unskip}\/:\space
      \ttfamily##2\par}\endgroup}%
  \def\urladdr##1##2{\begingroup
    \def~{\char`\~}%
    \@ifnotempty{##2}{\nobreak\noindent\urladdrname
      \@ifnotempty{##1}{, \ignorespaces##1\unskip}\/:\space
      \ttfamily##2\par}\endgroup}%
  \addresses
  \endgroup
  \global\let\addresses=\@empty
}
\def\@setabstracta{%
    \ifvoid\abstractbox
  \else
    \skip@25\p@ \advance\skip@-\lastskip
    \advance\skip@-\baselineskip \vskip\skip@
    \box\abstractbox
    \prevdepth\z@ % because \abstractbox is a vtop
    \vskip-10pt
  \fi
}
\renewenvironment{abstract}{%
  \ifx\maketitle\relax
    \ClassWarning{\@classname}{Abstract should precede
      \protect\maketitle\space in AMS document classes; reported}%
  \fi
  \global\setbox\abstractbox=\vtop \bgroup
    \normalfont\small
    \list{}{\labelwidth\z@
      \leftmargin0pc \rightmargin\leftmargin
      \listparindent\normalparindent \itemindent\z@
      \parsep\z@ \@plus\p@
      
    }%
    \item[\hskip\labelsep\bfseries\abstractname.]%
}{%
  \endlist\egroup
  \ifx\@setabstract\relax \@setabstracta \fi
}
\def\section{\@startsection{section}{1}%
  \z@{-1.5\linespacing\@plus-.2\linespacing}{.8\linespacing}%
  {\normalfont\bfseries\Large}}
\def\subsection{\@startsection{subsection}{2}%
  \z@{-.8\linespacing\@plus-.3\linespacing}{.3\linespacing\@plus.2\linespacing}%
  {\normalsize\bfseries}}
\def\subsubsection{\@startsection{subsection}{3}%
  \z@{.7\linespacing\@plus.2\linespacing}{-1.5ex}%
  {\normalfont\itshape}}
\def\@secnumfont{\bfseries}
\makeatother
\fi %\iftrue/false for amsart.cls hack

\def\to{\mathchoice{\longrightarrow}{\rightarrow}{\rightarrow}{\rightarrow}}
\makeatletter
\newcommand{\shortxra}[2][]{\ext@arrow 0359\rightarrowfill@{#1}{#2}}
\def\longrightarrowfill@{\arrowfill@\relbar\relbar\longrightarrow}
\newcommand{\longxra}[2][]{\ext@arrow 0359\longrightarrowfill@{#1}{#2}}
\renewcommand{\xrightarrow}[2][]{\mathchoice{\longxra[#1]{#2}}%
  {\shortxra[#1]{#2}}{\shortxra[#1]{#2}}{\shortxra[#1]{#2}}}
\makeatother

\makeatletter
\def\Nopagebreak{\@nobreaktrue\nopagebreak}
\makeatother

\numberwithin{equation}{section}

\theoremstyle{plain}
\newtheorem{theorem}{Theorem}[section]
\newtheorem{proposition}[theorem]{Proposition}

\newtheorem{lemma}[theorem]{Lemma}

\newtheorem*{claim}{Claim}
\theoremstyle{definition}
\newtheorem{definition}[theorem]{Definition}
\newtheorem{question}[theorem]{Question}

\newtheorem{remark}[theorem]{Remark}

\numberwithin{equation}{section}

\newtheoremstyle{TheoremNum}
        {}{}              %%% space between body and thm
        {\itshape}                      %%% Thm body font
        {}                              %%% Indent amount (empty = no indent)
        {\bfseries}                     %%% Thm head font
        {.}                             %%% Punctuation after thm head
        { }                             %%% Space after thm head
        {\thmname{#1}\thmnote{ \bfseries #3}}%%% Thm head spec
\theoremstyle{TheoremNum}

\newtheorem{thmn}{Theorem}

\def\Z{\mathbb{Z}}
\def\Q{\mathbb{Q}}
\def\R{\mathbb{R}}

\def\cP{\mathcal{P}}

\def\Im{\operatorname{Im}}

\def\Tor{\operatorname{Tor}}
\def\Ext{\operatorname{Ext}}

\def\Arf{\operatorname{Arf}}

\def\rhot{\rho^{(2)}}
\def\wt{\widetilde}

\def\setminus{\smallsetminus}

\def\Bor{\text{Bor}}
\def\mathbinover#1#2{\mathbin{\mathop{#1}\limits_{#2}}}

\def\cupover#1{\mathbinover{\cup}{#1}}

\def\SL{\operatorname{SL}}

\newcommand{\eps}{\varepsilon}

\let\oldsharp=\# \def\#{\mathbin{\oldsharp}}

\def\toiso{\xrightarrow{\simeq}}

\DeclareMathOperator{\Id}{Id}
\def\fr{\mathrm{fr}}
\DeclareMathOperator{\pr}{pr}
\DeclareMathOperator{\qr}{qr}
\DeclareMathOperator{\cl}{cl}

\def\emptystr{}
\newcommand{\mkc}[2][]{\begin{color}{red}#2%
  \def\tempstr{#1}%
  \ifx\tempstr\emptystr \else\textsf{\SMALL\ \raise.7ex\hbox{[\tempstr]}}\fi
\end{color}}

 \def\ol{\overline}
 \def\ba{\begin{array}}
 \def\ea{\end{array}}

 \def\sm{\setminus}

\begin{document}

\title
[Non-concordant links with homology cobordant zero surgeries]
{Non-concordant links with homology cobordant zero framed surgery manifolds}

\author{Jae Choon Cha}
\address{
  Department of Mathematics\\
  POSTECH\\
  Pohang 790--784\\
  Republic of Korea\\
  and\linebreak
  School of Mathematics\\
  Korea Institute for Advanced Study \\
  Seoul 130--722\\
  Republic of Korea
}
\email{jccha@postech.ac.kr}

\author{Mark Powell}
\address{
  Department of Mathematics\\
  Indiana University \\
  Bloomington, IN 47405\\
  USA
}
\email{macp@indiana.edu}

\begin{abstract}

  We use topological surgery theory to give sufficient conditions for
  the zero framed surgery manifold of a 3-component link to be
  homology cobordant to the zero framed surgery on the Borromean rings
  (also known as the 3-torus), via a topological homology cobordism
  preserving the free homotopy classes of the meridians.

  This enables us to give examples of 3-component links with
  unknotted components and vanishing pairwise linking numbers, such
  that any two of these links have homology cobordant zero surgeries
  in the above sense, but the zero surgery manifolds are not
  homeomorphic.  Moreover the links are not concordant to one another,
  and in fact they can be chosen to be height $h$ but not height $h+1$
  symmetric grope concordant, for each $h$ which is at least three.

\end{abstract}

\def\subjclassname{\textup{2010} Mathematics Subject Classification}
\expandafter\let\csname subjclassname@1991\endcsname=\subjclassname
\expandafter\let\csname subjclassname@2000\endcsname=\subjclassname

\maketitle

\section{Introduction}

It is well known that the study of homology cobordism of 3-manifolds
is essential for understanding the concordance of knots and links: homology
cobordism of the \emph{exteriors} of links in $S^3$ is equivalent to
concordance in a homology $S^3\times I$, and an additional mild normal
generation condition for $\pi_1$ is equivalent to topological
concordance in $S^3\times I$ (this also holds modulo the 4-dimensional Poincar\'e
conjecture in the smooth case).

We recall the definitions: two $m$-component links $L_0$ and $L_1$ in
$S^3$ are said to be \emph{topologically} (respectively
\emph{smoothly}) \emph{concordant} if there exist $m$ locally flat
(respectively smoothly embedded) disjoint annuli in $S^3\times [0,1]$
cobounded by components of $L_0\times \{0\}$ and $-L_1\times \{1\}$.
Two 3-manifolds $M_0$ and $M_1$ bordered by a 2-manifold $\Sigma$,
that is endowed with a marking $\mu_i\colon
\Sigma\xrightarrow{\simeq} \partial M_i$, are \emph{topologically}
(respectively \emph{smoothly}) \emph{homology cobordant} if there is a
topological (respectively smooth) 4-manifold $W$ with $\partial W =
M_0 \sqcup -M_1 \sqcup \Sigma \times [0,1] / (\mu_0(x)\sim x \times
\{0\},\mu_1(x) \sim x \times \{1\}$, $x\in \Sigma)$, such that the
inclusions $M_i\to W$ ($i=0,1$) induce isomorphisms on integral
homology groups.  In this paper links are oriented, and link exteriors
are always bordered by $\bigsqcup_m S^1\times S^1$ under the
zero framing.

In high dimensions, concordance classification results were obtained
by studying homology surgery, with the aim of surgeries being to
produce a homology cobordism of the exteriors (for example,
see~\cite{Cappell-Shaneson:1974-1,Cappell-Shaneson:1980-1,LeDimet:1988-1}).
On the other hand, for knots and links in dimension three, the
\emph{zero surgery manifolds} and their 4-dimensional homology
cobordisms have been extensively used in the literature in order to
understand the structure peculiar to low dimensions, especially in the
topological category.  Recall that performing zero framed surgery on a
link in $S^3$ yields a closed $3$-manifold, called the zero surgery
manifold.

The classical invariants such as the knot signature and
Levine's algebraic knot concordance
class~\cite{Levine:1969-2, Levine:1969-1} are obtained from the zero surgery manifold of a knot, via the Blanchfield form.
Also higher order knot concordance obstructions, such as Casson-Gordon
invariants~\cite{Casson-Gordon:1978-1, Casson-Gordon:1986-1}, and
Cochran-Orr-Teichner
$L^2$-signatures~\cite{Cochran-Orr-Teichner:1999-1} are obtained from
the zero surgery manifold (often together with the homology class of
the meridian).

A natural interesting question is whether the homology cobordism class
of a zero surgery manifold determines the concordance class of a knot
or link or if it determines the homology cobordism class of the
exterior.

In this paper we show, in a strong sense involving homotopy of
meridians, that the answer is negative for a large class of
\emph{links} satisfying a certain nonvanishing condition on Milnor's
$\overline{\mu}$-invariants, even in the framework of symmetric grope
and Whitney tower generalisations of concordance and homology
cobordism in the sense of
\cite{Cochran-Orr-Teichner:1999-1,Cha:2012-1}.  Also we employ
topological surgery in dimension 4 to give a new construction of
homology cobordisms of zero surgery manifolds.  Next we state our main
theorems, after which we will discuss these aspects further.

\begin{theorem}
  \label{theorem:links-not-conc-hom-cobordant-0-surgeries}
  Suppose $h\ge 3$.  Then there are infinitely many 3-component links
  $L_0$, $L_1$, \ldots\ with vanishing pairwise linking numbers and with
  unknotted components, satisfying the following for any $i\ne j$.
  \begin{enumerate}
  \item\label{item:thm-links-not-conc-hom-cob-I-non-diffeo} The zero
    surgery manifolds $M_{L_i}$ and $M_{L_j}$ are not homeomorphic.
  \item\label{item:thm-links-not-conc-hom-cob-II-top-hom-cobordant}
    There is a topological homology cobordism between $M_{L_i}$ and $M_{L_j}$
    in which the $k$th meridians of $L_i$ and $L_j$ are freely homotopic for
    each~$k=1,2,3$.
  \item\label{item:thm-links-not-conc-hom-cob-III-grope-concordance}
    The links $L_i$ and $L_j$ are height $h$ but not height $h+1$
    symmetric grope concordant.  In particular $L_i$ and $L_j$ are not
    concordant.
  \end{enumerate}
\end{theorem}

For a definition of height $h$ symmetric grope concordance, see
Definition~\ref{Definition:grope-concordance}.  Our links are obtained
from the Borromean rings, which will be our $L_0$ for all $h$, by
performing a satellite construction along a curve lying in the kernel
of the inclusion induced map $\pi_1(S^3\sm {L_0})\to \pi_1(M_{L_0})$.

As a counterpoint to
Theorem~\ref{theorem:links-not-conc-hom-cobordant-0-surgeries}, we
show that there are infinite families of links with the same
non-vanishing Milnor invariants with homeomorphic zero surgery
manifolds preserving the homotopy classes of the meridians, but which
are not concordant.

The Milnor invariant of an $m$-component link associated to a
multi-index $I=i_1i_2\cdots i_r$ with $i_j\in \{1,\ldots,m\}$, as
defined in~\cite{Milnor:1957-1}, will be denoted by
$\overline{\mu}_L(I)$.  We denote its length by $|I|:=r$.  Define
$k(m) := \lfloor \log_2 (m-1)\rfloor$.

\begin{theorem}\label{theorem:links-diffeo-zero-surgeries-not-concordant}
  Let $I$ be a multi-index with non-repeating indices with length~$m
  \ge 2$.  For any $h \geq k(m)+2$ there are infinitely many
  $m$-component links $L_0,L_1,...$ with unknotted components,
  satisfying the following.
  \begin{enumerate}
  \item\label{item-thm-diff-zero-surg-one} The $L_i$ have identical
    $\ol\mu$-invariants: $\ol{\mu}_{L_i}(J) = \ol{\mu}_{L_j}(J)$
    for all~$J$.  In addition $\ol{\mu}_{L_i}(I) = 1$, and
    $\ol{\mu}_{L_i}(J) = 0$ for $|J|<|I|$.
  \item\label{item-thm-diff-zero-surg-two} There is a homeomorphism
    between the zero surgery manifolds $M_{L_i}$ and $M_{L_j}$ which
    preserves the homotopy classes of the meridians.
  \item\label{item-thm-diff-zero-surg-three} The links $L_i$ and $L_j$
    are height $h$ but not height $h+1$ symmetric grope concordant.
    In particular $L_i$ and $L_j$ are not concordant.
  \end{enumerate}
\end{theorem}

The case of $m\ge 3$ should be compared with
Theorem~\ref{theorem:links-not-conc-hom-cobordant-0-surgeries} since
then the links $L_i$ have vanishing pairwise linking numbers.  To
construct such links we start with certain iterated Bing doubles
constructed using T.~Cochran's algorithm which realise the Milnor
invariant required.  We then perform satellite operations which affect
the concordance class of the link but do not change the homeomorphism
type of the zero surgery manifold.

We remark that we could
also phrase Theorems
\ref{theorem:links-not-conc-hom-cobordant-0-surgeries} and
\ref{theorem:links-diffeo-zero-surgeries-not-concordant} in terms of
symmetric Whitney tower concordance instead of grope concordance.

In the three subsections below, we discuss some features of
Theorem~\ref{theorem:links-not-conc-hom-cobordant-0-surgeries},
regarding: (i)~the use
of topological surgery in dimension~4, (ii)~link concordance versus zero surgery homology cobordism, and
(iii)~link exteriors and the homology surgery approach.

\subsection{Topological surgery for 4-dimensional homology cobordism}

An interesting aspect of the proof of
Theorem~\ref{theorem:links-not-conc-hom-cobordant-0-surgeries} is that
we employ topological surgery in dimension 4 to give a sufficient
condition for certain zero surgery manifolds of 3-component links to
be homology cobordant.  It is well known that topological surgery in
dimension 4 is useful for obtaining homology cobordisms (and
consequently concordances), although the current state of the art in
terms of ``good'' groups, for which the $\pi_1$-null disc lemma is
known, is still insufficient for the general case.  M.~Freedman and
F.~Quinn showed that knots of Alexander polynomial one are concordant
to the unknot~\cite[Theorem~11.7B]{Freedman-Quinn:1990-1}.  J.~Davis
extended the program to show that two component links with Alexander
polynomial one are concordant to the Hopf link~\cite{Davis:2006-1}.
The above two cases use topological surgery over fundamental groups
$\Z$ and $\Z^2$ respectively.  Due to the rarity of good groups for
4-dimensional topological surgery, there are not many other situations
where such positive results on knot and link concordance can currently
be proven.  As another case, S.~Friedl and P.~Teichner in
\cite{Friedl-Teichner:2005-1} found sufficient conditions for a knot
to be homotopy ribbon, and in particular slice, with a certain ribbon
group $\Z \ltimes \Z[1/2]$.

We give another instance of the utility of topological surgery for
constructing homology cobordisms, using the group $\Z^3$, which is
manageable from the point of view of topological surgery in
dimension~4.  Indeed, our sufficient condition for zero surgery
manifolds to be homology cobordant focuses on the Borromean rings as a
base link.  The zero surgery manifold $M_{\Bor}$ of the Borromean
rings is the 3-torus $T^3 = S^1\times S^1\times S^1$, whose
fundamental group is~$\Z^3$.

To state our result, we use the following notation: let $\Lambda :=
\Z[\Z^3]= \Z[t_1^{\pm 1},t_2^{\pm 1},t_3^{\pm 1}]$.  Denote the zero
surgery manifold of a link $L$ by $M_L$ as before.  For a
$3$-component link $L$ with vanishing pairwise linking numbers, there is a canonical homotopy class of maps
$f_L\colon M_L\to M_{\Bor}=T^3$ which send the homotopy class of the
$i$th meridian of $L$ to that of the Borromean rings, namely the $i$th
circle factor of $T^3$.  After choosing an identification of
$\pi_1(T^3) = \Z^3$, we can use this to define the
$\Lambda$-coefficient homology $H_1(M_L;\Lambda)$.  We say that a map
$f \colon M_L\to T^3$ is a \emph{$\Lambda$-homology equivalence} if
$f$ is homotopic to $f_L$ and $f$ induces isomorphisms on
$H_*(-;\Lambda)$.

\begin{theorem}\label{Theorem:borromean_homology_equiv}
  Suppose $L$ is a $3$-component link whose components have trivial
  Arf invariants and there exists a $\Lambda$-homology equivalence
  $M_L \to T^3$.  Then there is a homology cobordism $W$ between $M_L$
  and $T^3=M_{\Bor}$ for which the inclusion induced maps $\pi_1(M_L)\to
  \pi_1(W) \xleftarrow{\simeq} \pi_1(T^3)$ are such that the composition
  from left to right takes meridians to meridians.
\end{theorem}

\subsection{Link concordance versus zero surgery homology cobordism}

We review the general question of whether links with
homology cobordant zero surgery manifolds are concordant.  The answer
to the basic question is easily seen to be no, once one knows of a
result of C. Livingston that there are knots not concordant to their
reverses~\cite{Livingston:1983-1}.  Note that a knot and its reverse
have the same zero surgery manifold.  This leads us to consider some
additional conditions on the homology cobordism, involving the
meridians.  In what follows meridians are always positively oriented.

First, observe the following: the exteriors of two links are homology
cobordant if and only if the zero framed meridians cobound framed
annuli disjointly embedded in a homology cobordism of the zero surgery
manifolds.  (For the if direction, note that the exterior of the
framed annuli is a homology cobordism of the link exteriors.)  In
particular it holds if two links (or knots) are concordant.

Regarding the knot case,
in~\cite{Cochran-Franklin-Hedden-Horn:2011-1}, T.~Cochran,
B.~Franklin, M.~Hedden and P.~Horn considered homology cobordisms of
zero surgery manifolds in which the meridians are \emph{homologous}:
in the smooth category, they showed that the existence of such a
homology cobordism is insufficient for knots to be concordant.  In the
topological case this is still left unknown.

Concerning a stronger \emph{homotopy} analogue, the following is
unknown in both the smooth and topological cases:

\begin{question}
  If there is a homology cobordism of zero surgery manifolds of two
  knots in which the meridians are homotopic, are the knots concordant? Or
  concordant in a homology $S^3\times I$?
\end{question}

For the link case, results in the literature give non-concordant
examples whose zero surgery manifolds admit a homology cobordism with
homotopic meridians.  As a generic example in the topological
category, consider a 2-component link with linking number one.  The
zero surgery manifold is a homology 3-sphere, which bounds a
contractible topological 4-manifold
by~\cite[Corollary~9.3C]{Freedman-Quinn:1990-1}. Taking the connected
sum of such 4-manifolds, one obtains the following: \emph{the zero
  surgery manifolds of any two linking number one 2-component links
  cobound a simply connected topological homology cobordism.}  Note
that in this case the meridians are automatically homotopic.  There
are many linking number one 2-component links which are not
concordant, as can be detected, for example, by the multivariable
Alexander polynomial~\cite{Kawauchi:1978-1, Nakagawa-1978}.  For
related in-depth study, the reader is referred to, for
instance,~\cite{Cha-Ko:1999-2,Friedl-Powell:2011-1,Cha:2012-1}.  With
our respective coauthors, we gave non-concordant linking number one
links with two unknotted components, for which abelian invariants such
as the multivariable Alexander polynomial are unable to obstruct them
from being concordant.

There are other examples, which have knotted components: in~\cite[end
of Section~1]{Cochran-Franklin-Hedden-Horn:2011-1}, they discuss
2-component linking number zero links with homeomorphic zero surgery
manifolds which have non-concordant (knotted) components.  These links
are obviously not concordant, and it can be seen that the
homeomorphisms preserve meridians up to homotopy.

By contrast with the above examples, our links have unknotted
components and vanishing pairwise linking numbers.  Another feature
exhibited by the links of Theorems
\ref{theorem:links-not-conc-hom-cobordant-0-surgeries}
and~\ref{theorem:links-diffeo-zero-surgeries-not-concordant} is that a great deal of the subtlety of symmetric grope concordance of
links can occur within a single homology cobordism/homeomorphism
class of the zero surgeries, even modulo local knot tying.

We remark that all the links of
Theorems~\ref{theorem:links-not-conc-hom-cobordant-0-surgeries}
and~\ref{theorem:links-diffeo-zero-surgeries-not-concordant} lie in
the same `$k$-solvequivalence class' for all $k$ in the sense
of~\cite[Definition~2.5]{Cochran-Kim:2004-1}.

\subsection{Link exteriors and the homology surgery approach}

Our results serve to underline the philosophy that when investigating
the relative problem of whether two links are concordant, and neither
of them are the unlink, one should consider obstructions to homology
cobordism of the link exteriors viewed as bordered manifolds,
rather than to homology cobordism of the zero surgery manifolds,
even in low dimensions.  This was implemented in, for
example~\cite{Kawauchi:1978-1, Nakagawa-1978, Cha:2012-1}
(see also \cite{Friedl-Powell:2011-1} for a related approach).

Although we stated our results in terms of grope concordance of links,
in Theorems~\ref{theorem:links-not-conc-hom-cobordant-0-surgeries}
and~\ref{theorem:links-diffeo-zero-surgeries-not-concordant} given
above, in fact we show more: the link exteriors are far from being homology
cobordant, as measured in terms of Whitney towers.  A more detailed
discussion is given in Section~\ref{section:amenable-signature}.  For
the purpose of distinguishing exteriors, we use the amenable
Cheeger-Gromov $\rho$-invariant technology for bordered 3-manifolds
(particularly for link exteriors) developed in~\cite{Cha:2012-1},
generalising applications of $\rho$-invariants to concordance and
homology cobordism in~\cite{Cochran-Orr-Teichner:1999-1,
  Cochran-Harvey-Leidy:2009-1, Cha-Orr:2009-01}.

We will now discuss our results from the viewpoint of the homology
surgery approach to link concordance classification, initiated by
S. Cappell and
J. Shaneson~\cite{Cappell-Shaneson:1974-1,Cappell-Shaneson:1980-1} and
implemented in high dimensions by J. Le Dimet~\cite{LeDimet:1988-1}
using P. Vogel's homology localisation of spaces~\cite{Vogel:1978-1}.
The strategy consists of two parts.  Consider the problem of comparing
two given link exteriors.  First we decide whether the exteriors have
the same ``Poincar\'e type,'' which roughly means that they have
homotopy equivalent Vogel homology localisations.  If so, there is a
common finite target space, into which the exteriors are mapped by
homology equivalences rel.\ boundary.  Once this is the case, a
surgery problem is defined, and one can try to decide whether homology
surgery gives a homology cobordism of the exteriors.  The first step
is obstructed by homotopy invariants (including Milnor
$\overline{\mu}$-invariants in the low dimension).  The failure of the
second step is measured by surgery obstructions, which are not yet
fully formulated in the low dimension (even modulo that 4-dimensional
surgery might not work), since the fundamental group plays a more
sophisticated central r\^{o}le; see \cite{Powell:2012-1} for the
beginning of an algebraic surgery approach to this problem in the
context of knot slicing.

Our examples illustrate that for many Poincar\'e types, namely those
in Theorems~\ref{theorem:links-not-conc-hom-cobordant-0-surgeries}
and~\ref{theorem:links-diffeo-zero-surgeries-not-concordant}, we get a
rich theory of surgery obstructions within each Poincar\'e type, which
is invisible via zero surgery manifolds.  We remark that for our links
$L_i$ in
Theorems~\ref{theorem:links-not-conc-hom-cobordant-0-surgeries}
and~\ref{theorem:links-diffeo-zero-surgeries-not-concordant}, there is
a homology equivalence of the exterior of each $L_i$ into that of a
fixed one, say~$L_1$, since we use satellite constructions (see
Section~\ref{section:construction-links-grope-concordance}).  It
follows that the exteriors have the same Poincar\'e type in the above
sense.  In this paper, (parts of the not-yet-fully-formulated)
homology surgery obstructions in dimension~4 have their incarnation in
Theorem~\ref{theorem:amenable-signature-for-solvable-cobordism}, the
Amenable Signature Theorem.

\subsection*{Organisation of the paper}

In Section~\ref{section:hypotheses}, we explore the implications of
the hypothesis that a homology equivalence $f \colon M_L \to T^3$ as
in Theorem \ref{Theorem:borromean_homology_equiv} exists, and we prove
Theorem \ref{Theorem:borromean_homology_equiv} in Section
\ref{section:proof-of-thm-1.2}.  In
Section~\ref{section:construction-links-grope-concordance}, we
construct links with a given Milnor invariant with non-repeating
indices, and perform satellite operations on the links to construct
the links of Theorems
\ref{theorem:links-not-conc-hom-cobordant-0-surgeries}
and~\ref{theorem:links-diffeo-zero-surgeries-not-concordant}, which
are height $h$ symmetric grope concordant.  In
Section~\ref{section:amenable-signature}, we show that none of these
links are height $h+1$ grope concordant to one another.

\subsection*{Acknowledgements}

We would like to thank Stefan Friedl for many conversations: our first
examples were motivated by joint work with him.  We also thank Jim
Davis for discussions relating to the proof of his theorem on Hopf
links, and Prudence Heck, Charles Livingston, Kent Orr and Vladimir
Turaev for helpful discussions.
Finally we thank the referee for valuable comments and for being impressively expeditious.

Part of this work was completed while the first author was a visitor
at Indiana University in Bloomington and the second author was a
visitor at the Max Planck Institute for Mathematics in Bonn.  We would
like to thank these institutions for their hospitality.

The first author was partially supported by National Research
Foundation of Korea grants 2013067043 and 2013053914.  The second
author gratefully acknowledges an AMS Simons travel grant which aided
his travel to Bonn.

\section{Homology type of zero surgery manifolds and the 3-torus}
\label{section:hypotheses}

This section discusses the hypotheses of
Theorem~\ref{Theorem:borromean_homology_equiv}.  We begin the section
by briefly reminding the reader who is familiar with Kirby calculus of
a nice way to see the following well known fact.

\begin{lemma}\label{lemma:zero-surg-borromean-equals-3-torus}
  The zero surgery manifold of the Borromean rings is homeomorphic to
  the 3-torus.
\end{lemma}

\begin{proof}
Place dots on two components of the Borromean
rings and a zero near the other.  Each component of the Borromean
rings is a commutator in the meridians of the other two components,
so this is a Kirby diagram for $T^2 \times D^2$, whose boundary is
$T^3$.  The 1-handles (dotted circles) can be replaced with
 zero framed 2-handles without changing the boundary.
\end{proof}

In the following proposition we expand on the meaning and implications
of the condition in Theorem \ref{Theorem:borromean_homology_equiv}.
Denote the exterior of a link $L$ by $X_L := S^3 \setminus \nu L$ as
before.

\begin{proposition}\label{Prop:necessary-conditions-for-f}

  Suppose that $L$ is a 3-component link.  Then the following are
  equivalent.
  \begin{enumerate}
  \item \label{item-nec-conditions-1} There is a $\Lambda$-homology
    equivalence $f\colon M_L \to T^3$.
  \item \label{item-nec-conditions-2} The preferred longitudes generate the link
    module $H_1(X_L;\Lambda)$.
  \item \label{item-nec-conditions-3} The pairwise linking numbers of $L$ vanish and
    $H_1(M_L;\Lambda)=0$.
  \end{enumerate}
  Furthermore, (any of) the above conditions imply that $L$ has
  multi-variable Alexander polynomial
  $\Delta_L=(t_1-1)(t_2-1)(t_3-1)$, and
  $\Delta_L=(t_1-1)(t_2-1)(t_3-1)$ implies that the Milnor invariant
  $\ol{\mu}_L(123)$ is equal to $\pm 1$.
\end{proposition}

\begin{proof}
  First we will observe (\ref{item-nec-conditions-2}) and
  (\ref{item-nec-conditions-3}) are equivalent.  Longitudes of $L$
  represent elements in $H_1(X_L;\Lambda) \cong
  \pi_1(X_L)^{(1)}/\pi_1(X_L)^{(2)}$ if and only if they are zero in
  $H_1(X_L;\Z)\cong \Z^3$; that is, the pairwise linking numbers are
  zero.  If this is the case, $H_1(M_L;\Lambda)$ is isomorphic to
  $H_1(X_L;\Lambda)/\langle\text{longitudes}\rangle$, since $M_L$ is
  obtained by attaching three 2-handles to $E_L$ along the longitudes
  and then attaching three 3-handles along the boundary.  It follows
  that longitudes generate $H_1(X_L;\Lambda)$ if and only if
  $H_1(M_L;\Lambda)=0$.

  Suppose (\ref{item-nec-conditions-1}) holds.  Denote the meridians
  of $L$ by $\mu_i$ ($i=1,2,3$) and the linking number of the $i$th
  and $j$th components by $\ell_{ij}$.  The $i$th longitude
  $\lambda_i$, which is homologous to $\sum_{j\ne i} \ell_{ij}\mu_i$,
  is zero in $H_1(M_L;\Z)\cong H_1(T^3;\Z)$.  Since $\{f_*([\mu_i])\}$
  forms a basis of $H_1(T^3;\Z)\cong \Z^3$, it follows by linear
  independence that $\ell_{ij}=0$ for any $i$ and~$j$.  Also,
  $H_1(M_L;\Lambda) \cong H_1(T^3;\Lambda)=0$.  This shows
  that~(\ref{item-nec-conditions-3}) holds.

  Suppose (\ref{item-nec-conditions-3}) holds.  Start with a map
  $g\colon \partial X_L = \bigsqcup_3 S^1\times S^1 \to T^3$ that
  sends $\mu_i$ to the $i$th $S^1$ factor and $\lambda_i$ to a point.
  Observe that $g_*\colon H_1(\partial X_L;\Z) \to H_1(T^3;\Z)$
  factors through the inclusion induced map $i_*\colon H_1(\partial
  X_L;\Z) \to H_1(M_L;\Z)$ and the identifications $H_1(M_L;\Z)\toiso
  \Z^3 \xleftarrow{\simeq} H_1(T^3;\Z)$; this follows from the fact
  that $H_1(\partial X_L;\Z) \cong \Z^6$ is generated by the $\mu_i$
  and $\lambda_i$ and that both $g_*$ and $i_*$ are quotient maps,
  with their kernels generated by the~$\lambda_i$.  Since $T^3$ is a
  $K(\Z^3,1)$, elementary obstruction theory shows that $g$ extends to
  a map $f\colon M_L \to T^3$.

  Consider the universal coefficient spectral sequence (see e.g.~\cite[Theorem~2.3]{Levine:1977-1})
  $E^2_{p,q}=\Ext^p_\Lambda(H_q(M_L;\Lambda),\Lambda) \Rightarrow
  H^n(M_L;\Lambda)$.  We have $E^2_{0,1}=0$ since
  $H_1(M_L;\Lambda)=0$, and $E^2_{1,0} =
  \Ext^1_\Lambda(\Z,\Lambda)=H^1(T^3;\Lambda)=0$.  It follows that
  $H^1(M_L;\Lambda)=0$.  By duality, $H_2(M_L;\Lambda)=0$.  Also,
  $H_3(M_L;\Lambda)=0$ since the $\Z^3$-cover of $M_L$ is non-compact.
  Since $H_0(M_L;\Lambda)\cong \Z \cong H_0(T^3;\Lambda)$ and
  $H_i(T^3;\Lambda)=0$ for $i>0$, it follows that $f$ is a
  $\Lambda$-homology equivalence.  This completes the proof of the
  equivalence of (\ref{item-nec-conditions-1}),
  (\ref{item-nec-conditions-2}) and~(\ref{item-nec-conditions-3}).

  Suppose (\ref{item-nec-conditions-1}), (\ref{item-nec-conditions-2})
  and (\ref{item-nec-conditions-3}) hold.  Recall that the scalar
  multiplication of a loop by $t_i$ in the module $H_1(X_L;\Lambda)$
  is defined to be conjugation by~$\mu_i$.  Since $\lambda_i$ and
  $\mu_i$ commute as elements of the fundamental group, we have
  $(t_i-1)\lambda_i=0$ in $H_1(X_L;\Lambda)$.  From this and
  (\ref{item-nec-conditions-2}), it follows that there is an
  epimorphism of $A:=\bigoplus_{i=1}^3 \Lambda/\langle t_i-1\rangle$
  onto $H_1(X_L;\Lambda)$.  Since the zero-th elementary ideal of $A$
  is the principal ideal generated by $(t_1-1) (t_2-1) (t_3-1)$, it
  follows that $\Delta_L$ is a factor of $(t_1-1) (t_2-1) (t_3-1)$.
  We now invoke the Torres condition (see e.g.\
  \cite[Theorem~7.4.1]{Kawauchi:1996-1}): $\Delta_L(1,t_2,t_3) =
  (t_2^{\ell_{12}}t_3^{\ell_{13}}-1)\Delta_{L'}(t_2,t_3)$ where $L'$
  is the sublink of $L$ with the first component deleted and
  $\ell_{ij}$ is the pairwise linking number.  Since $\ell_{ij}=0$ by
  (3), we have $\Delta_L(1,t_2,t_3)=0$.  It follows that $t_1-1$ is a
  factor of $\Delta_L$.  Similarly $t_2-1$ and $t_3-1$ are factors.
  Therefore $\Delta_L(t_1,t_2,t_3)=(t_1-1)(t_2-1)(t_3-1)$.

  To show the last part, suppose that $\Delta_L(t_1,t_2,t_3) =
  (t_1-1)(t_2-1)(t_3-1)$.  By
  \cite[Proposition~7.3.14]{Kawauchi:1996-1}, the single-variable
  Alexander polynomial $\Delta_L(t)$ of $L$ is given by
  \[
  \Delta_L(t) = (t-1)\Delta_L(t,t,t) = (t-1)^4 \doteq
  ((\sqrt{t})^{-1}-\sqrt{t})^4.
  \]
  It follows that $L$ has Conway polynomial $\nabla_L(z)=z^4$, by
  the standard substitution $z=(\sqrt{t})^{-1}-\sqrt{t}$.
  In~\cite[Theorem~5.1]{Cochran:1985-1}, Cochran identified the
  coefficient of $z^4$ in $\nabla_L(z)$ with $(\mu_L(123))^2$ for
  3-component links with pairwise linking number zero.  Applying this
  to our case, it follows that $\ol{\mu}_L(123) = \pm 1$.
\end{proof}

\section{Construction of homology cobordisms using topological surgery}
\label{section:proof-of-thm-1.2}

This section gives the proof of Theorem
\ref{Theorem:borromean_homology_equiv}.  The proof will use surgery
theory, and will parallel the proof given by Davis
in~\cite{Davis:2006-1} (see also \cite[Section~7.6]{Hillman:2002-1}).
We will provide some details in order to fill in where the treatment
in \cite{Davis:2006-1} was terse, and to convince ourselves that the
analogous arguments work in the case of interest.

For the convenience of the reader we restate
Theorem~\ref{Theorem:borromean_homology_equiv} here.

\begin{thmn}[\ref{Theorem:borromean_homology_equiv}]
  Suppose $L$ is a $3$-component link whose components have trivial
  Arf invariants and there exists a $\Lambda$-homology equivalence
  $M_L \to T^3$.  Then there is a homology cobordism $W$ between $M_L$
  and $T^3=M_{\Bor}$ for which the inclusion induced maps $\pi_1(M_L)\to
  \pi_1(W) \xleftarrow{\simeq} \pi_1(T^3)$ are such that the composition
  from left to right takes meridians to meridians.
\end{thmn}

\begin{remark}
  It is an interesting question to determine whether there are extra
  conditions which can be imposed in order to see that the Arf
  invariants of the components are forced to vanish by the homological
  assumptions.  In the cases of knots and two component links with
  Alexander polynomial one, the Arf invariants of the components are
  automatically trivial.  For knots $\Delta_K(-1)$ computes the Arf
  invariant by~\cite{Levine:1966-1}.  For two component links one
  observes that $\Delta_L(t,1)$ and $\Delta_L(1,t)$ give the Alexander
  polynomials of the components, by the Torres condition, and then
  applies Levine's theorem.  These arguments do not seem to extend to
  the three component case of current interest.
\end{remark}

The proof of Theorem~\ref{Theorem:borromean_homology_equiv} will
occupy the rest of this section.  In order to produce a homology
cobordism, we will first show that there is a normal cobordism between
normal maps $f \colon M_L \to T^3$ and $\Id \colon T^3 \to T^3$.
Interestingly, we can work with smooth manifolds in order to establish
the existence of a normal cobordism.  This will make arguments which
invoke tangent bundles and transversality easier to digest.  Only at
the end of the proof of Theorem
\ref{Theorem:borromean_homology_equiv}, where we take connected sums
with the $E_8$-manifold, and where we claim that the vanishing of a
surgery obstruction implies that surgery can be done, do we need to
leave the realm of smooth manifolds.

\begin{definition}
  Let $X$ be an $n$-dimensional manifold with a vector bundle $\nu \to
  X$.  A degree one normal map $(F,b)$ over $X$ is an $n$-manifold $M$
  with a map $F \colon M \to X$ which induces an isomorphism $F_*
  \colon H_n(M;\Z) \toiso H_n(X;\Z)$, together with a stable
  trivialisation $b \colon TM \oplus F^*\nu \oplus \eps^l \cong
  \eps^k$.

  A degree one normal cobordism $(J,e)$ between normal maps $(F \colon
  M \to X,b)$ and $(G \colon N \to X,c)$ is an $(n+1)$--dimensional
  cobordism between $M$ and $N$ with a map $J \colon W \to X \times I$
  extending $F \colon M \to X \times \{0\}$ and $G \colon M \to X
  \times \{1\}$, which induces an isomorphism $$J_* \colon
  H_{n+1}(W,\partial W;\Z) \toiso H_{n+1}(X \times I,X
  \times\{0,1\};\Z),$$ together with a stable trivialisation $e \colon
  TW \oplus J^*(\nu \times I) \oplus \eps^{l'} \cong \eps^{k'}$.
\end{definition}

For us, let $X = T^3$, and let $\nu$ be its tangent bundle.  We fix a
framing on the stable tangent bundle of the target torus $T^3$ once
and for all.  Note that this canonically determines a trivialisation
of the tangent bundle of $F^*\nu$, for any map $F \colon M \to X$, by
the following diagram, in which the bottom composition is the constant
map, denoted $\ast$, and the top composition is the pull back
$F^*\nu$.  The middle composition is the induced framing.
\[
\xymatrix @C+1cm{
M \times \{0\} \ar[d] \ar[r]^{F \times \Id} & T^3 \times \{0\} \ar[r]^{\nu} \ar[d] & BO(n) \ar[d] \\
M \times I \ar[r]^{F \times \Id} & T^3 \times I  \ar[r] & BO \\
M \times \{1\} \ar[u] \ar[r]^{F \times \Id} & T^3 \times \{1\} \ar[r]^{\ast} \ar[u] & BO(n). \ar[u]
}
\]
A framing of the tangent bundle of the domain will therefore determine
a normal map.

\begin{lemma}\label{Lemma:existence-of-normal-cobordism}
  Let $L$ be a link whose components have trivial $\Arf$ invariants,
  and let $f \colon M_L \to T^3$ be a degree one normal map which
  induces a $\Z$-homology isomorphism and which maps a chosen meridian
  $\mu_i$ to the $i$th $S^1$ factor of $T^3$ for $i=1,2,3$.  We can
  make a homotopy of $f$ and choose a framing on $M_L$ so that $f
  \colon M_L \to T^3$ and $\Id \colon T^3 \to T^3$ are degree one
  normal cobordant.
\end{lemma}

\begin{proof}
  We need to show that we can choose a framing on $M_L$ such that the
  disjoint union $M_L \sqcup -T^3$ represent the trivial element of
  $\Omega_3^{\fr}(T^3)$.  We compute this bordism group:
  \[
  \wt{\Omega}_3^{\fr}(T^3) \cong \wt{\Omega}_4^{\fr}(\Sigma T^3) \cong
  \wt{\Omega}_4^{\fr}(S^2 \vee S^2 \vee S^2 \vee S^3 \vee S^3 \vee S^3
  \vee S^4),
  \]
  with this last isomorphism induced by a homotopy equivalence of
  spaces.  There is a copy of $S^{i+1}$ for each $i$--cell of $T^3$,
  for $i =1,2,3$.  To see this homotopy equivalence, we need to see
  that the attaching maps of the cells are null-homotopic.  The
  suspension of the 1-skeleton of $T^3$ is $S^2 \vee S^2 \vee S^2$.
  The Hilton-Milnor theorem~\cite[Theorem~A]{Hilton-55} computes the
  homotopy groups of a wedge of spheres.  The attaching maps for the
  $2$-cells of $T^3$ become the attaching maps for the $3$-cells of
  $\Sigma T^3$, namely maps in $$\pi_2(S^2 \vee S^2 \vee S^2) \cong
  \bigoplus_3 \pi_2(S^2) \cong \bigoplus_3 \Z,$$ where the first
  isomorphism is by the Hilton-Milnor theorem.  The commutator
  attaching maps become trivial in the abelian $\pi_2(S^2)$.
  Therefore the 3-skeleton of $\Sigma T^3$ is a wedge~$S^2 \vee S^2
  \vee S^2 \vee S^3 \vee S^3 \vee S^3$.  The attaching map for the
  $3$-cell of $T^3$ becomes the attaching map for the $4$-cell of
  $\Sigma T^3$, a map in
  \[
  \pi_3(S^2 \vee S^2 \vee S^2 \vee S^3 \vee S^3 \vee S^3) \cong
  \bigoplus_{1 \le i \le 3} \pi_3(S^3) \oplus \bigoplus_3 \pi_3(S^2)
  \oplus \bigoplus_{1 \le i < j \le 3} \pi_3(S^3),
  \]
  again by the Hilton-Milnor theorem, where the last three
  $\pi_3(S^3)$ summands include into $\pi_3(S^2 \vee S^2 \vee S^2 \vee
  S^3 \vee S^3 \vee S^3)$ via composition with the Whitehead product:
  let $f_i \colon S^2 \to S^2_i$ be a generator of $\pi_2(S^2_i)$,
  where $S^2_i$ is the $i$th $S^2$ component in the wedge.  Then the
  Whitehead product is the homotopy class of the map $[f_i,f_j] \in
  \pi_3(S^2_i \vee S^2_j)$, which is the attaching map for the
  $4$-cell in a standard cellular decomposition of $S^2 \times S^2$.
  Since $\pi_2(S^1) \cong \pi_2(S^1 \vee S^1) \cong 0$, the summands
  associated to the $S^2$ components of the wedge do not arise from a
  suspension.  The summands associated to the $S^3$ components are
  null-homotopic since the 3-cell of $T^3$ is attached to each
  $2$-cell twice, once on either side.  This completes the explanation
  of the claimed homotopy equivalence:
  \[
  \Sigma T^3 \simeq S^2 \vee S^2 \vee S^2 \vee S^3 \vee S^3 \vee S^3
  \vee S^4.
  \]

  By Mayer-Vietoris, the bordism group~$\wt{\Omega}_4^{\fr}(S^2 \vee
  S^2 \vee S^2 \vee S^3 \vee S^3 \vee S^3 \vee S^4)$ is a direct sum
  \begin{align*}
    &\bigoplus_3\, \wt{\Omega}_4^{\fr} (S^2) \oplus \bigoplus_3
    \wt{\Omega}_4^{\fr}(S^3) \oplus \wt{\Omega}_4^{\fr} (S^4)
    \\
    \cong &\bigoplus_3\, \wt{\Omega}_2^{\fr} (S^0) \oplus \bigoplus_3
    \wt{\Omega}_1^{\fr} (S^0) \oplus \wt{\Omega}_0^{\fr} (S^0)
    \\
    \cong &\bigoplus_3\, \Omega_2^{\fr} \oplus \bigoplus_3
    \Omega_1^{\fr} \oplus \Omega_0^{\fr}
    \\
    \cong &\bigoplus_3\, \Z_2 \oplus \bigoplus_3 \Z_2 \oplus \Z.
  \end{align*}
  Therefore
  \[
  \Omega_3^{\fr}(T^3) \cong \Omega_3^{\fr} \oplus \bigoplus_3\, \Z_2
  \oplus \bigoplus_3 \Z_2 \oplus \Z \cong \Z_{24} \oplus \bigoplus_3\,
  \Z_2 \oplus \bigoplus_3 \Z_2 \oplus \Z.
  \]
  The isomorphism is given as follows.  Let
  \[
  \pr_i \colon T^3 = S^1 \times S^1 \times S^1 \to S^1
  \]
  be given by projection onto the $i$th factor.  Similarly let
  \[
  \qr_i \colon T^3 = S^1 \times S^1 \times S^1 \to S^1 \times S^1
  \]
  be given by forgetting the $i$th factor.  Let $F \colon M \to T^3$
  be an element of $\Omega_3^{\fr} (T^3)$. Making all maps transverse
  to a point, called $\ast$, we obtain an $8$-tuple:
  \begin{align*}
    \Big([M],(\pr_1 \circ F)^{-1}(\ast),(\pr_2 \circ F)^{-1}(\ast),
    (\pr_3 \circ F)^{-1}(\ast), (\qr_1 \circ F)^{-1}(\ast),
    \\
    (\qr_2 \circ F)^{-1}(\ast),(\qr_3 \circ F)^{-1}(\ast),
    F^{-1}(\ast) \Big)
    \\
    \in \Omega_3^{\fr} \oplus \bigoplus_3\, \Omega_2^{\fr} \oplus
    \bigoplus_3 \Omega_1^{\fr} \oplus \Omega_0^{\fr} \cong \Z_{24}
    \oplus \bigoplus_3\, \Z_2 \oplus \bigoplus_3 \Z_2 \oplus \Z.
  \end{align*}
We consider each of the summands in turn.

By choosing the appropriate orientation on $M_L$, and making the degree one normal maps transverse to a point, it can be arranged so that the disjoint union  $f^{-1}(\ast) \sqcup -\Id^{-1}(\ast) = \{pt\} \sqcup -\{pt\} = 0 \in \Omega_0^{\fr}$.

As observed in \cite[Proof~of~the~Lemma]{Davis:2006-1}, we can change the framing so that the elements of $\Omega_1^{\fr}$ agree.  First, we change the framing on each of three chosen meridians $\mu_i$ to the link components $L_i$.

Orientable $k$-plane vector bundles over $S^1$ are classified by homotopy classes of maps $[S^1,BSO(k)]$.  Consider the exact sequence:
\[\pi_2(BSO) \to \pi_2(BSO,BSO(k)) \to \pi_1(BSO(k)) \xrightarrow{\gamma} \pi_1(BSO).\]
A stably trivial vector bundle over $S^1$ gives us an element of $\ker(\gamma)$.  A choice of trivialisation of the vector bundle gives us a null homotopy and therefore an element of $\pi_2(BSO,BSO(k))$.  The possible choices of stable trivialisations, or framings, are indexed by $\pi_2(BSO) \cong \pi_1(SO) \cong \Z_2$.

We can therefore, if necessary, change the framing on each $\mu_i$ to be the bounding framing using an element of $\pi_1(SO(2))$ which maps to the non-trivial element of $\pi_1(SO)$.  Use the element of $\pi_1(SO(2))$ to change the framing on the normal bundle of $\mu_i$ in $M_L$.  We claim that these changes in the framing can be extended to the whole of $M_L$.  To see this, we argue as follows.  The dual of the inclusion map $H^1(M_L;\Z) \to H^1(\mu_i;\Z)$ is surjective, since each $[\mu_i]$ is a generator of $H_1(M_L;\Z)$.  The change of framing map $\mu_i \to SO(2)$ represents a homotopy class of maps in $[\mu_i, S^1]$ and therefore an element of $H^1(\mu_i;\Z)$.  Since this pulls back to an element of $H^1(M_L;\Z)$, which can in turn produce a map $M_L \to SO(2)$, the change of framing map can be extended as claimed.

Let $N_i \subset M_L$ be the submanifolds given by $(\qr_i \circ f)^{-1}(\ast)$, after perturbing $f$ to make $\qr_i \circ f$ transverse to a point.  As the inverse image of the $i$th $S^1$ factor of $T^3$ (e.g. $f^{-1}(S^1 \times \{\ast\} \times \{\ast\})$), $N_i$ is a collection of circles.  After a homotopy of $f$, it can be arranged, by the assumption on $f$, that $N_i$ is a single meridian $\mu_i$, which has the bounding framing and therefore represents the zero element in $\Omega_1^{\fr}$.  To make this arrangement, it suffices to be able to remove circles $N_i$ whose image in $T^3$ is null homologous.  But in $T^3$, a null homologous curve is also null homotopic.  Therefore we can make a homotopy of $f$ so that $N_i$ misses $S^1 \times \{\ast\} \times \{\ast\}$.

After another homotopy, the inverse image $(\pr_i \circ f)^{-1}(\ast)$ can
be arranged to be a capped-off Seifert surface $F_i \cup D^2$, where $F_i$ is a Seifert surface for $L_i$ (possibly with closed connected components).  To see this, we again use our assumption that $f$ sends the $i$th meridian $\mu_i$ to the $i$th
circle.  This assumption enables us to homotope $f$ so that $\pr_i \circ f$ sends a regular
neighbourhood $\mu_i\times D^2$ to $S^1$ by projection onto the first
factor.  Then the inverse image is as desired.  A homotopy of~$f$ preserves
the framed bordism class of $(\pr_i \circ f)^{-1}(\ast)$, and the
class $[F_i \cup D^2] \in \Omega_2^{\fr}$ is determined by the Arf
invariant of $L_i$.  By hypothesis, this vanishes.

Finally, again following \cite{Davis:2006-1} (see also \cite[Proof~of~Lemma~11.6B]{Freedman-Quinn:1990-1}), the framing can be altered in the neighbourhood of a point to change the element $[M] \in \Omega_3^{\fr}$ to the trivial element. We recall the definition of the $J$ homomorphism $J \colon \pi_3(SO) \to \pi_3^S \cong \Omega_3^{\fr}$, for the convenience of the reader, where $\pi_k^S$ is the $k$th stable homotopy group of spheres.
(Incidentally, $\pi_3(SO) \cong \Z$ and $\pi_3^S \cong \Z_{24}$.)   Given
$\theta \colon S^3 \to SO$, choose a $k$ sufficiently large so that we can represent $\theta$ by a map $\theta \colon S^3 \to SO(k)$.  We proceed to construct a map $(J(\theta) \colon S^{k+3} \to S^k) \in \pi_3^S$.  So:
\[S^{k+3} = S^3 \times D^k \cup_{S^3 \times D^{k-1}} D^4 \times S^{k-1}.\]
Define a map:
\[\begin{array}{rcl} j(\theta)\colon S^3 \times D^k &\to& S^3 \times D^k \\
(x,y) & \mapsto& (x,\theta(x)(y)),\end{array}\]
since $\theta(x) \in SO(k)$ acts on $D^k$ by identifying $D^k$ with the unit ball in $\R^k$.  This map extends to a homeomorphism $j(\theta)$ of $S^3 \times D^k \cup_{S^3 \times D^{k-1}} D^4 \times S^{k-1}$.  Form the composition:
\begin{align*} S^{k+3} &= S^3 \times D^k \cup_{S^3 \times D^{k-1}} D^4 \times S^{k-1} \xrightarrow{j(\theta)} S^3 \times D^k \cup_{S^3 \times D^{k-1}} D^4 \times S^{k-1} \\ &\xrightarrow{col} S^3 \times S^k \xrightarrow{proj_1} S^k,\end{align*}
where $col$ is the collapse map which squashes $D^4 \times S^{k-1}$ and $proj_1$ is the projection onto the first factor.  This gives an element of $\pi_3^S$, which is the image of $\theta$ under $J\colon \pi_3(SO) \to \pi_3^S \cong \Omega_3^{\fr}$.

This $J$ homomorphism is onto~\cite[Example~7.17]{Adams:1966-1}, so that composing the framing in a neighbourhood $D^3$ of a point with the choice of map in $\theta \in \pi_3(SO) = [(D^3,\partial D^3),(SO,\ast)]$ such that $-J(\theta) = [M] \in \Omega_3^{\fr}$, changes the class in $\Omega_3^{\fr}$ as desired.

This shows the existence of a normal cobordism $W'$.  To see that this is of degree one, note that the map to $T^3$ which extends over $W'$ can be used to define a map to $T^3 \times I$, by defining a map $g \colon W' \to I$ such that $g(M_L) =\{0\}$ and $g(T^3) = \{1\}$.   Now consider the commutative diagram:
\[\xymatrix{H_4(W',\partial W';\Z) \ar[r] \ar[d] & H_3(\partial W';\Z) \ar[d] \\ H_4(T^3 \times I, T^3 \times\{0,1\};\Z) \ar[r] & H_3(T^3 \times \{0,1\};\Z).}\]
Going right, then down, the fundamental class $[W',\partial W']$ maps to $$(1,-1) \in H_3(T^3 \times \{0,1\};\Z) \cong \Z \oplus \Z.$$  By commutativity, the relative fundamental class $[W',\partial W']$ must map to a generator of $H_4(T^3 \times I, T^3 \times\{0,1\};\Z)$
\end{proof}

A $\Lambda$-homology equivalence is also an integral homology
equivalence, by the following argument.  By definition, (see above the
statement of Theorem~\ref{Theorem:borromean_homology_equiv}), a
$\Lambda$-homology equivalence induces an isomorphism on $H_1(-;\Z)$.
By duality, we also have an isomorphism on $H_2(-;\Z)$.  It remains to
see that $f \colon M_L \to T^3$ is a degree one map.
The assumption that $f_* \colon H_*(M_L ;\Lambda) \toiso
  H_*(T^3;\Lambda)$ is an isomorphism implies that the relative
  homology vanishes: $H_*(T^3,M_L;\Lambda) \cong 0$.  The universal
  coefficient spectral sequence then implies that $H_*(T^3,M_L;\Z)
  \cong 0$ since all the $E^2$ terms
  $\Tor^{\Lambda}_p(H_q(T^3,M_L;\Lambda),\Z)$ vanish.  Therefore a $\Lambda$-homology equivalence as in Theorem~\ref{Theorem:borromean_homology_equiv} is a degree one map.

Lemma~\ref{Lemma:existence-of-normal-cobordism} then establishes the existence of a choice of stable framing $b$ on $M_L$ such that there is a degree one normal cobordism
\[(F' \colon W' \to T^3 \times I, e')\]
between $(f \colon M_L \to T^3,b)$ and $(\Id\colon T^3 \to T^3, c)$.  Choosing such a framing, we proceed to apply surgery theory to alter $W'$ into a homology cobordism.  Davis' observation in~\cite{Davis:2006-1} was that the framing on $W'$ is not an intrinsic part of the concordance problem, but rather necessary additional data which is required in order to be able to apply surgery theory.  Without the information provided by the self intersection form, it is not possible to obtain algebraic sufficient conditions which ensure that surgery can be performed.  Nevertheless, as we shall see, there is a certain amount of freedom in the choice of framing data.

Before giving the proof of Theorem \ref{Theorem:borromean_homology_equiv}, we first give the definition of the Wall even dimensional surgery obstruction groups, which we will use in the proof.

\begin{definition}[\cite{Wall-Ranicki:1999-1}, Chapter~5]
\label{definition:L-groups-even}
  Let $A$ be a ring with involution.
  A \emph{$(-1)^k$-Hermitian sesquilinear quadratic form} on a free based $A$-module $M$ is a $(-1)^k$-Hermitian sesquilinear form $\lambda \colon M \times M \to A$ together with a quadratic enhancement.
  A \emph{quadratic enhancement} of a form $\lambda \colon M \times M \to A$ is a function
  $\mu \colon M \to A/\{a - (-1)^k \overline{a} \, | \, a \in A\}$ such that
  \begin{enumerate}
    \item $\lambda(x,x) = \mu(x) + \overline{\mu(x)}$;
    \item $\mu(x+y) -\mu(x) - \mu(y) = \lambda(x,y)$;
    \item $\mu(ax) = a \mu(x) \overline{a}$;
  \end{enumerate}
 for all $x,y \in M$ and for all $a \in A$.

A hyperbolic quadratic form is a direct sum of standard hyperbolic forms, where the standard hyperbolic form $(H,\chi,\nu)$ is given by \[\Big(A \oplus A, \begin{pmatrix}0 & 1 \\ (-1)^k & 0  \end{pmatrix}, \nu((1,0)^T) = 0 = \nu((0,1)^T)\Big).\]

The even dimensional surgery obstruction group $L_{2k}(A)$ is defined to be the Witt group of nonsingular $(-1)^k$-Hermitian sesquilinear quadratic forms on free based $A$-modules, where addition in the Witt group is by direct sum, and the equivalence class of the hyperbolic forms is the identity element, where the equivalence relation is as follows.  Quadratic forms $(M,\lambda,\mu)$ and $(M',\lambda',\mu')$ are said to be equivalent if there are hyperbolic forms $(H, \chi, \nu)$ and $(H',\chi',\nu')$ such that there is an isomorphism of forms $(M \oplus H,\lambda \oplus \chi, \mu \oplus \nu) \cong (M' \oplus H',\lambda' \oplus \chi', \mu' \oplus \nu')$.
This completes the definition of $L_{2k}(A)$.
\end{definition}

For us $A$ will be the group ring $\Z[\pi]$ of some group $\pi$; initially $\pi$ will be $\Z^3$ so that we take $A = \Z[\Z^3] = \Lambda$. We omit the definition of the odd dimensional $L$-groups since they will only play a peripheral r\^{o}le in the proof of Theorem \ref{Theorem:borromean_homology_equiv}.

\begin{proof}[Proof of Theorem \ref{Theorem:borromean_homology_equiv}]

First, do surgery below the middle dimension \cite[Chapter~1]{Wall-Ranicki:1999-1} on $(W',F',e')$ to create a normal cobordism $(F\colon W \to T^3 \times I, e)$ which is $2$-connected i.e. $W$ is connected and $\pi_1(W) \cong \pi_1(T^3) \cong \Z^3$.  The induced map $F_*\colon \pi_2(W) \to \pi_2(T^3 \times I)$ is automatically surjective since $T^3$ is aspherical.

The Wall surgery obstruction \cite[Chapter~5]{Wall-Ranicki:1999-1} of the normal cobordism $(F\colon W \to T^3 \times I, e)$ is now defined in $L_4(\Z[\Z^3])$, to be given by the intersection form
\[\lambda_{W'} \colon H_2(W';\Lambda) \times H_2(W';\Lambda) \to \Lambda,\]
together with the quadratic enhancement
\[\mu \colon H_2(W;\Lambda) \to \Z[\Z^3]/\{a=\ol{a}\}\]
defined by counting the self intersections of an immersion of a sphere $S^2\looparrowright W$ representing an element of $H_2(W;\Lambda) \cong \pi_2(W)$, where the regular homotopy class of the immersion is fixed by the framing $e$ to be the unique class of immersions for which the induced trivialisation of $TS^2$ extends over the null-homotopy of $S^2$ in $T^3$.

The fact that the homology of the boundary $H_j(M_L;\Lambda) \cong H_j(T^3;\Lambda) \cong 0$ for $j = 1,2$, is used crucially here to see that the intersection form $\lambda_W$ is non-singular, as observed by the surgeon in the ``dialogue'' of \cite{Davis:2006-1}.

By \cite[Proposition~13B.8]{Wall-Ranicki:1999-1}, which is based on Shaneson's formula $L_n(\Z[\pi \times \Z]) \cong L_n(\Z[\pi]) \oplus L_{n-1}(\Z[\pi])$, when $\pi$ has trivial Whitehead group \cite{Shaneson:1969-1}, we have that:
\begin{align*} L_4(\Z[\Z^3]) & \cong \bigoplus_{i=0}^3 \begin{pmatrix} 3 \\ i \end{pmatrix} L_{4-i}(\Z) \cong L_4(\Z) \oplus \bigoplus_3 L_3(\Z) \oplus \bigoplus_3 L_2(\Z) \oplus L_1(\Z) \\  & \cong L_0(\Z) \oplus \bigoplus_3 L_2(\Z),\end{align*}
where the last isomorphism is by periodicity of the $L$-groups and the fact that the odd dimensional simply connected $L$-groups vanish.  The even dimensional simply-connected $L$-groups $L_{2k}(\Z)$ are computed~\cite{Kervaire-Milnor:1963-1}, when $k = 0 \;\text{mod}\, 2$, as \[\begin{array}{rcl} L_0(\Z) &\toiso  & \Z \\(M,\lambda,\mu) &\mapsto & \sigma(\R \otimes_\Z M, \Id \otimes \lambda)/8, \end{array}\] while for the dimensions where $k = 1 \;\text{mod}\, 2$ they are computed via \[\begin{array}{rcl} L_2(\Z) &\toiso &\Z_2 \\  (M,\lambda,\mu)& \mapsto & \Arf(\Z_2 \otimes_{\Z} M, \Id \otimes \lambda, \Id \otimes \mu). \end{array}\]

We need to see that we can make further alterations to $W$ in order to make the surgery obstruction vanish.

First, we take the connected sum with $-\sigma(W)/8$ copies of the $E_8$ manifold, namely the simply connected $4$-manifold which is constructed by plumbing disc bundles $D^2 \times D^2$ according to the $E_8$ lattice.  It turns out that the boundary of the resulting 4-manifold is the Poincar\'{e} homology sphere.  One then caps off with the contractible topological 4-manifold whose boundary is the Poincar\'{e} homology sphere~\cite[Corollary~9.3C]{Freedman-Quinn:1990-1}.  This produces the $E_8$ manifold, a closed topological 4-manifold.   It has a non-singular intersection form, with a quadratic enhancement induced from a normal map to $S^4$, and its signature is~$8$.  By a negative copy of this 4 manifold we of course mean the same manifold but with the opposite choice of orientation.  By making such a modification to $W$, we obtain a new normal map, which by abuse of notation we again denote by $(W,F,e)$, for which the obstruction in $L_0(\Z)$ is trivial.  Note that $W$ s
 till has fundamental group $\Z^3$ since $\pi_1(E_8 \text{ manifold}) \cong \{1\},$ and moreover $\partial W$ is unchanged.

Next, we may need to alter $W$ again, so that the three $\Arf$ invariant obstructions in $L_2(\Z)$ vanish.  For $i=1,2,3$, define maps
\[\qr_i \colon T^3 \times I = S^1 \times S^1 \times S^1 \times I \to S^1 \times S^1\]
which forget the $i$th $S^1$ factor and the $I$ factor.  Perform a homotopy of $F$ to ensure that $\qr_i\circ F$ is transverse to $\ast \in S^1 \times S^1$, and such that $F^{-1}(S^1 \times \{\ast\} \times \{\ast\} \times \{\partial I\}) \to S^1 \times \{\ast\} \times \{\ast\} \times \{\partial I\}$ is a homotopy equivalence (and similarly with the $\ast$ terms moved appropriately for $i=2,3$). This homotopy equivalence was already arranged in the proof of Lemma \ref{Lemma:existence-of-normal-cobordism}, when we saw that the elements of $\Omega_1^{\fr}$ can be removed.  Let $S_i$ be the surface
$(\qr_i \circ F)^{-1}(\ast)$; each surface has boundary $\partial S_i$ given by the meridian $\mu_i$ and the corresponding $S^1$ factor of~$T^3$.

Let $\pr_i \colon T^3 \times I = S^1 \times S^1 \times S^1 \times I \to S^1 \times I$ be the map which remembers the $i$th $S^1$ factor and the $I$ factor.  Making $F$ transverse to a point, $(\pr_i \circ F)^{-1}(\ast)$ is a surface $\Sigma_i \subset W$.  Since $F(S_i \cap \Sigma_i)$ is a single point and $F$ is of degree one, we can assume that $S_i$ and $\Sigma_i$ intersect in a single point.  By choosing different points in the $I$ factor, we can ensure that the $\Sigma_i$ are all distinct.

Now as in \cite{Davis:2006-1}, for each $i$ with nonzero surgery obstruction in the corresponding $L_2(\Z)$ summand of $L_4(\Z[\Z^3])$, remove a neighbourhood $\Sigma_i \times D^2$ of $\Sigma_i$ and replace it with $\Sigma_i \times \cl(S^1 \times S^1 \setminus D^2)$.  That is, replace the $D^2$ factor with a punctured torus, but define the framing on the torus to be the framing which yields $\Arf$ invariant one, that is the Lie framing on both $S^1$ factors.  Since $\Sigma_i$ is dual to $S_i$, this adds one to the $\Arf$ invariant of the element of $L_2(\Z)$ represented by $S_i$, and so changes the $\Arf$ invariant one summands to having $\Arf$ invariant zero.

After these alterations we have a normal map $(G' \colon V' \to T^3 \times I,k')$, with vanishing surgery obstruction.  Since the fundamental group $\Z^3$ is \emph{good} in the sense of Freedman (poly-cyclic groups are good~\cite[Theorem~5.1A]{Freedman-Quinn:1990-1}), the surgery sequence is exact in the topological category---see \cite[Theorem~11.3A]{Freedman-Quinn:1990-1}.  We can therefore find embedded two spheres representing a half-basis for $\pi_2(G')$, perform surgery, and obtain a topological $4$-manifold $V$ which is homotopy equivalent to $T^3 \times I$; in particular, $V$ is a homology cobordism between $M_L$ and $T^3$.

Moreover, the following diagram commutes.
\[\xymatrix{
\pi_1(M_L) \ar[r] \ar[d]_{f_*} & \pi_1(V) \ar[d]^{\cong} & \pi_1(T^3) \ar[l] \ar[d]^{\cong}_{\Id} \\
\pi_1(T^3) \ar[r]^-{\cong} & \pi_1(T^3 \times I) & \pi_1(T^3) \ar[l]_-{\cong}
}\]
Since the meridians~$\mu_i$ of~$L$ are mapped to standard generators of $\pi_1(T^3)$, an easy diagram chase shows that the homotopy classes of the meridians are preserved in the homology cobordism $V$.
\end{proof}

\section{Construction of links and grope concordance}
\label{section:construction-links-grope-concordance}

In this section we give constructions of certain links with a given
Milnor invariant, and construct grope concordances, using the methods of
\cite{Cochran:1990-1} and~\cite{Cha:2012-1}.

\subsection{Iterated Bing doubles with a prescribed Milnor invariant}
\label{subsection:links-with-given-mu-invariant}

Let $I$ be a multi-index with non-repeating indices with
length~$m:=|I|\ge 2$.  We describe a rooted binary tree $T(m)$
associated to~$m\ge 2$, which has $m$ leaves: the right subtree of the
root just consists of a single vertex, and the left subtree $T^{\dag}(m)$ is the
complete binary tree of height $h(m):=\lceil\log_2(m-1)\rceil$ with
the rightmost $2(m-h(m)-1)$ pairs of leaves (and edges ending at
these) removed.  (By convention, a binary tree is always embedded in
a plane with the root on the top.)  That is, $T^{\dag}(m)$ is a minimal
height binary tree with $m-1$ leaves.  For example, $T(m)$ for $m=7$
is shown in Figure~\ref{figure:symm-tree}.

\begin{figure}[htb]
  \def\blt{\vcenter to 0mm{\vss\hbox to 0mm{\hss$\bullet$\hss}\vss}}
  \def\bltl#1{\vtop{\hbox{$\vcenter to 0mm{\vss\hbox to
          0mm{\hss$\bullet$\hss}\vss}$}\hbox to
      0mm{\hss\raise1ex\hbox{\small$#1$}\hss}}}
  \def\bltlv#1{\vtop{\hbox to 0mm{$\vcenter to 0mm{\vss\hbox to
          0mm{\hss$\bullet$\hss}\vss}$}\vbox to 0mm{\vss\hbox to
        0mm{\hss\raise1ex\hbox{\small$#1$}\hss}}}}
  \[
  \xymatrix@M=0mm@W=0mm@H=0mm@R=2em@C=1.5em{
    &&&&&&&&&&&\blt \ar@{-}[lllld] \ar@{-}[rd] \\
    &&&&&&&\blt \ar@{-}[rrd]\ar@{-}[lllld]&&&&&\bltlv{7}\\
    &&&\blt \ar@{-}[rrd]\ar@{-}[lld]&&&&&&\blt \ar@{-}[rd]\ar@{-}[ld]\\
    &\blt \ar@{-}[rd]\ar@{-}[ld]&&&&\blt \ar@{-}[rd]\ar@{-}[ld]&&&\bltlv{5}&&\bltlv{6}\\
    \bltl{1}&&\bltl{2}&&\bltl{3}&&\bltl{4} \\
  }
  \]
  \caption{The tree $T(m)$ for $m=7$, labeled with $I=1234567$.}
  \label{figure:symm-tree}
\end{figure}
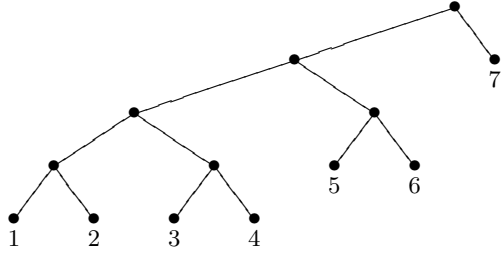

Following the proof of \cite[Theorem~7.2]{Cochran:1990-1}, a rooted
binary tree $T$ describes a link with components corresponding to the
leaves of~$T$. First, a complete binary tree of height one is
associated to a Hopf link.  If $T$ is obtained from $T'$ by attaching
two new leaves to a leaf $v$ of $T'$, then the link associated to $T$
is obtained from that of $T$ by Bing doubling the component
corresponding to~$v$.

Consider the link described by the tree~$T(m)$.  Labelling the leaves
of $T(m)$ from left to right with the multi-index $I$ (see
Figure~\ref{figure:symm-tree} for $I=1234567$), the components of the
link are ordered.  We denote this ordered link by~$L_I$.  Then, by
\cite[Theorem~8.1]{Cochran:1990-1}, the link $L_I$ has
$\ol{\mu}_{L_I}(I)=\pm 1$ and $\ol{\mu}_{L_I}(I') =0$ for $|I'| <|I|$.

\subsection{Satellite construction and grope concordance of links}
\label{subsection:satellite-construction-grope-concordance}

To construct links which are grope concordant, we employ the method
of~\cite[Section~4]{Cha:2012-1}.  We begin by giving the definition of
grope concordance.  The use of gropes in this context first appeared
in \cite{Cochran-Orr-Teichner:1999-1}.

\begin{definition}[\cite{Freedman-Teichner:1995-1}]
  A \emph{grope} is a pair (2-complex, base circle) of a certain type described below.
  A grope has a height $h \in \mathbb{N}$. For $h = 1$ a grope is precisely a
  compact oriented surface $\Sigma$ with a single boundary component
  which is the base circle. A grope of height $h+1$ is defined
  inductively as follows: let $\{\alpha_i\,|\, i = 1, \ldots, 2
  \cdot\text{genus}\}$ be a standard symplectic basis of circles
  for~$\Sigma$. Then a grope of height $h+1$ is formed by attaching
  gropes of height $h$ to each $\alpha_i$ along the base circles.

  An annular grope is defined by replacing the bottom stage surface
  by a surface with two boundary components.
\end{definition}

\begin{definition}[\cite{Cha:2012-1}, Definition~2.16]
  \label{Definition:grope-concordance}
  Two $m$-component links $L$ and $L'$ in $S^3$ are \emph{height $n$
    grope concordant} if there are $m$ framed annular gropes $G_i$ of
  height $n$, $i=1,\dots m$, disjointly embedded in $S^3 \times
  [0,1]$, with the boundary of $G_i$ the zero framed $i$th component
  of $L_i \subset S^3 \times \{0\}$ and $-L'_i \subset S^3 \times
  \{1\}$.
\end{definition}

As mentioned in the introduction we could also phrase our theorems in
terms of Whitney towers, but for simplicity of exposition we stick to
gropes.  See \cite[Section~2]{Cha:2012-1} for an exposition on gropes,
Whitney towers, and $n$-solvable cobordisms (Our
Section~\ref{section:amenable-signature} also contains a limited
discussion of $n$-solvable cobordisms).

We recall that a \emph{capped grope of height $k$} is a grope of
height $k$ together with 2-discs attached along each standard
symplectic basis curves of the top layer surfaces.  The attached
2-discs are called \emph{caps}, and the grope itself is called the
\emph{body}.  We always assume that a capped grope embedded in a
4-manifold is framed.

We denote the exterior of a link $L$ by~$X_L$.  If $L$ is a link in
$S^3$, $\eta$ is an unknotted circle in $S^3$ disjoint from $L$, and
$K$ is a knot, then we denote the satellite link of $L$ with axis
$\eta$ and companion $K$ by $L(\eta,K)$; this is the image of $L$
under the homeomorphism $X_{\eta} \cup_{\partial} X_K
\xrightarrow{\approx} S^3,$ where the glueing identifies the longitude
of $\eta$ with the meridian of $K$, and vice versa.

Following \cite[Definition~4.2]{Cha:2012-1}, we call $(L,\eta)$ a
\emph{satellite configuration of height $k$} if $L$ is a link in
$S^3$, $\eta$ is an unknotted circle in $S^3$ disjoint from $L$, and
the 0-linking parallel of $\eta$ in $X_\eta=X_\eta\times \{0\}$ bounds
a capped grope of height $k$ embedded in $X_\eta\times [0,1]$ with
body disjoint from $L\times[0,1]$.  The caps should be embedded in
$X_\eta\times [0,1]$ but may intersect $L \times [0,1]$.

Lemma~\ref{lemma:grope-concordance-for-iterated-satellite-construction}
stated below describes how iterated satellite constructions using
satellite configurations give us grope concordant links.  The setup is
as follows.  Fix~$n$.  (To obtain
Theorems~\ref{theorem:links-not-conc-hom-cobordant-0-surgeries}
and~\ref{theorem:links-diffeo-zero-surgeries-not-concordant}, set $h=n+2$.)
  Suppose that $(L_0,\eta)$ is a satellite configuration
of height $k\le n$.  (Later we will use the link $L_I$ described above
as~$L_0$.)  Suppose that $(K_i,\alpha_i)$ is a satellite configuration
of height one, with $K_i$ a slice knot, for $i=0,\ldots,n-k-1$.  Let
$J_0^j$ be the connected sum of $N_j$ copies of the knot described in
\cite[Figure~3.6]{Cochran-Teichner:2003-1}, where $\{N_j\}$ is an
increasing sequence of integers which will be specified later.
(Indeed these will be given in terms of the Cheeger-Gromov bound on
the $\rho$-invariants and, for the links of
Theorem~\ref{theorem:links-not-conc-hom-cobordant-0-surgeries}, in
terms of the Kneser-Haken bound on the number of disjoint non-parallel
incompressible surfaces.  See
Section~\ref{section:amenable-signature}, just before the proof of
Proposition~\ref{theorem:being-non-solvably-cobordant} and
Section~\ref{subsection:non-diffeomorphic-zero-surgery}, just before
Lemma~\ref{lemma:non-homeomorphic-zero-surgery}.)  Define $J_i^j :=
K_{i-1}(\alpha_{i-1},J_{i-1}^j)$ inductively for $i=1,\ldots, n-k$.
Finally define $L_j := L_0(\eta,J_{n-k}^j)$.

\begin{lemma}[{\cite[Proposition~4.7]{Cha:2012-1}}]
  \label{lemma:grope-concordance-for-iterated-satellite-construction}
  The link $L_j$ is height $n+2$ grope concordant to $L_0$ for all~$j$.
\end{lemma}

\begin{proof} The same as the proof of
  \cite[Proposition~4.7]{Cha:2012-1}, except that $L_0$ replaces the
  Hopf link in the last sentence.
\end{proof}

The following observation on the satellite construction is useful.

\begin{lemma}\label{lemma:milnor-invariants-unchanged-satellite}
  If $L'=L(\eta,K)$ is obtained from $L$ by a satellite construction,
  then $L$ and $L'$ have the same Milnor
  $\ol{\mu}$-invariants.
\end{lemma}

\begin{proof}
  It is well known that a satellite construction $L'=L(\eta,K)$ comes
  with an integral homology equivalence $f\colon (X_{L'},\partial
  X_{L'}) \to (X_L,\partial X_L)$ which restricts to a homeomorphism
  on the boundary preserving longitudes and meridians (see e.g.\
  \cite[Proof of Proposition~4.8]{Cha:2007-1},
  \cite[Lemma~5.3]{Cha-Orr:2011-01}).  As in
  \cite[Lemma~2.1]{Cha-Friedl-Powell:2012-1}, by
  Stallings~\cite{Stallings:1965-1}, it follows that $f$ induces an
  isomorphism $\pi_1(X_L)/\pi_1(X_L)_q \cong
  \pi_1(X_{L'})/\pi_1(X_{L'})_q$ that preserves the classes of
  meridians and longitudes for any~$q$, and consequently $L$ and $L'$
  have identical $\ol{\mu}$-invariants.
\end{proof}

\subsection{Satellite configuration of iterated Bing doubles}
\label{subsection:satellite-configuration-of-iterated-Bing-doubles}

The goal of this section is to construct the examples needed for Theorem~\ref{theorem:links-diffeo-zero-surgeries-not-concordant}, and to show that they satisfy (\ref{item-thm-diff-zero-surg-one}) and (\ref{item-thm-diff-zero-surg-two}) of that theorem.
Now we consider again the link $L_I$ described in
Section~\ref{subsection:links-with-given-mu-invariant}.  Recall that
$k(m) := \lfloor \log_2 (m-1)\rfloor$ where $m=|I|$.  Let $\eta$ be
the zero framed longitude of the component of $L_I$ labelled with $m$,
namely the component of the original Hopf link that is never Bing
doubled in the construction of~$L_I$.

\begin{lemma}\label{lemma:properties-of-L_I-and-eta}
  \leavevmode\Nopagebreak
  \begin{enumerate}
  \item\label{item:properties-of-L_I-eta-one} The pair $(L_I,\eta)$ is
    a satellite configuration of height $k(m)$.
  \item\label{item:properties-of-L_I-eta-two} The curve $\eta$ is
    nonzero in $\pi_1(L_I)/\pi_1(L_I)_{m}$.
  \item\label{item:properties-of-L_I-eta-three} For any knot $K$, the
    link $L_I(\eta,K)$ has zero surgery manifold homeomorphic to the
    zero surgery manifold of~$L_I$, via a homeomorphism which
    preserves the homotopy classes of the meridians.
  \end{enumerate}
\end{lemma}

We remark that
Lemma~\ref{lemma:properties-of-L_I-and-eta}~(\ref{item:properties-of-L_I-eta-two})
will be used in Section~\ref{section:amenable-signature}.

\begin{proof}
  Denote $L:=L_I$ for this proof.

  (1) We go back to the construction of $L$, and construct the grope
  as we construct $L$.  We begin with the Hopf link (i.e.\ $m=2$), and
  the curve $\eta$ as a longitude of $L_2$.  We also begin with a
  thickened cap $D^2 \times [-1,1]$, such that $\partial D^2 \times
  \{0\} = \eta$.  This intersects the other component of the Hopf link
  in a single point.

  Every time a component $K$ is Bing doubled in the construction of
  $L$, we arrange that one of the clasps lies in $D^2 \times [-1,1]$,
  and then replace the thickened cap that intersected $K$ with a genus
  one capped surface with a single boundary component, whose body
  surface misses the new Bing doubled components, and such that each
  cap intersects one of the two new components.  See
  Figure~\ref{figure:cap-capped-surface}, which is somewhat
  reminiscent of a figure in~\cite[Chapter~2.1]{Freedman-Quinn:1990-1}.

  \begin{figure}[htb]
    \includegraphics{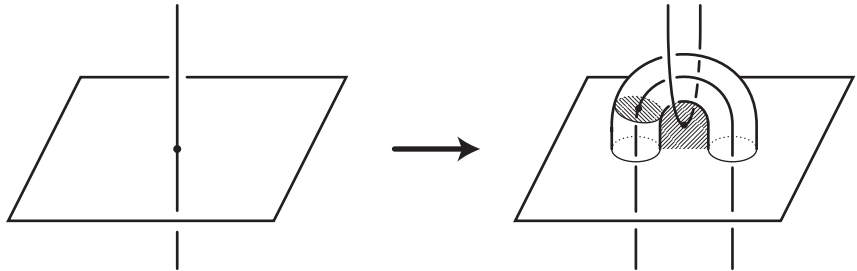}
    \caption{Replacing a cap with a capped surface}
    \label{figure:cap-capped-surface}
  \end{figure}

  Since a complete binary tree of height $k(m)$ can be embedded in
  $T(m)$, we obtain a symmetric embedded capped grope of the required
  height, with the body lying in the link exterior $X_L$ and the caps
  intersecting the link transversely.

  (2) The nonvanishing of the Milnor invariant $\ol{\mu}_{L}(I)$
  implies that all of the longitudes of $L$ are nontrivial in
  $\pi_1(X_L)/\pi_1(X_L)_{|I|}$.

  (3) A Kirby diagram for the 3-manifold $M_L$ given by zero framed
  surgery on $L$ can be produced by putting a $0$ next to every
  component of $L$.  If we perform a satellite construction with
  pattern $K$ and with $\eta$ as axis, this is equivalent to tying all
  the strands of $L$ which intersect a disc $D$, whose boundary is
  $\eta$, in the knot $K$, with framing zero.  In other words, replace
  the trivial string link in $D \times [0,1]$ with the string link
  obtained by taking suitably many parallel copies of~$K$.

  But we can make a crossing change of these parallel copies of $K$ at
  will, by performing handle slides, sliding the parallel strands
  over the zero framed 2-handle attached along the component parallel
  to~$\eta$.  This gives a Kirby presentation of a homeomorphic
  3-manifold.

  By making sufficiently many such crossing changes/handle slides, all
  the parallel strands which the satellite construction ties in the
  knot $K$ can be unknotted, recovering the link $L$.  Thus the zero
  surgery manifolds of the satellite link and the original link are
  homeomorphic.  It is easy to see that the homotopy classes of the
  meridians of $L$ are preserved under such
  homeomorphisms.\qedhere
\end{proof}

Now, let $n \ge k(|I|)=k(m)$.  Let $L_j$ be the links obtained by the
construction just before
Lemma~\ref{lemma:grope-concordance-for-iterated-satellite-construction},
using our $(L_I,\eta)$ as $(L_0,\eta)$, and using the Stevedore
satellite configuration described in \cite[Figure~6]{Cha:2012-1}, which for the reader's convenience is shown in Figure~\ref{figure:stevedore-conf}, as
the $(K_i,\alpha_i)$.  Then by
Lemma~\ref{lemma:grope-concordance-for-iterated-satellite-construction}
and
Lemma~\ref{lemma:properties-of-L_I-and-eta}~(\ref{item:properties-of-L_I-eta-one}),
the links $L_j$ are height $n+2$ grope concordant to the
link~$L_0=L_I$.

\begin{figure}[htb]
  \labellist
  \small\hair 0mm
  \pinlabel {$K_i$} at -8 30
  \pinlabel {$\alpha_i$} at 94 30
  \endlabellist
  \includegraphics{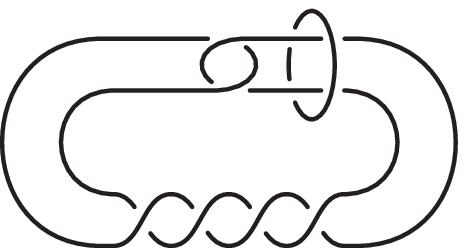}
  \caption{Stevedore satellite configuration $(K_i,\alpha_i)$.}
  \label{figure:stevedore-conf}
\end{figure}

Lemma~\ref{lemma:milnor-invariants-unchanged-satellite} shows that the
links $L_j$ satisfy
Theorem~\ref{theorem:links-diffeo-zero-surgeries-not-concordant}~(\ref{item-thm-diff-zero-surg-one}).
They also satisfy
Theorem~\ref{theorem:links-diffeo-zero-surgeries-not-concordant}~(\ref{item-thm-diff-zero-surg-two})
by
Lemma~\ref{lemma:properties-of-L_I-and-eta}~(\ref{item:properties-of-L_I-eta-three}).
We have also proved, in
Lemma~\ref{lemma:grope-concordance-for-iterated-satellite-construction},
the first part of
Theorem~\ref{theorem:links-diffeo-zero-surgeries-not-concordant}~(\ref{item-thm-diff-zero-surg-three}):
the links $L_j$ are mutually height $n+2$ grope concordant.  The
second part of
Theorem~\ref{theorem:links-diffeo-zero-surgeries-not-concordant}~(\ref{item-thm-diff-zero-surg-three}),
namely the failure of the links to be pairwise height $n+3$ grope
concordant, will be shown in Section~\ref{section:amenable-signature}.

\subsection{Examples with non-homeomorphic zero surgery manifolds}
\label{subsection:non-diffeomorphic-zero-surgery}

In order to produce examples satisfying
Theorem~\ref{theorem:links-not-conc-hom-cobordant-0-surgeries}~(\ref{item:thm-links-not-conc-hom-cob-I-non-diffeo}),
we alter the construction of
Sections~\ref{subsection:satellite-construction-grope-concordance} and
\ref{subsection:satellite-configuration-of-iterated-Bing-doubles} to
give examples with non-homeomorphic zero surgery manifolds.  We
consider the case of $m=3$ and $I=123$ only.  Then the link $L := L_I$
described in Section~\ref{subsection:links-with-given-mu-invariant} is
the Borromean rings.  Let $\eta$ be the simple closed curve in $S^3\sm
L$ shown in Figure~\ref{figure:borromean-conf}; $x$, $y$, and $z$
denote the components of~$L$.

\begin{figure}[htb]
  \labellist
  \small\hair 0mm
  \pinlabel {$x$} at 105 184
  \pinlabel {$y$} at 185 160
  \pinlabel {$z$} at 0 73
  \pinlabel {$\eta$} at 87 162
  \endlabellist
  \includegraphics{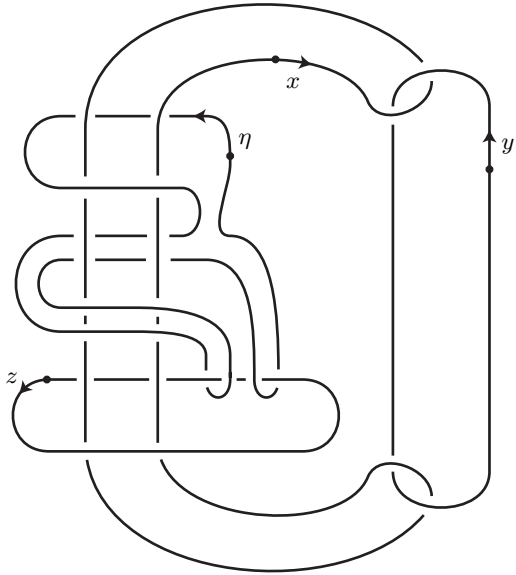}
  \caption{A satellite configuration on the Borromean rings}
  \label{figure:borromean-conf}
\end{figure}

The pair $(L,\eta)$ also has two of the properties stated in
Lemma~\ref{lemma:properties-of-L_I-and-eta}, for $m=3$:

\begin{lemma}\leavevmode\Nopagebreak
  \begin{enumerate}
  \item The pair $(L, \eta)$ is a satellite configuration of height
    one.
  \item In $\pi_1(X_L)$, $\eta=[x,y][[x,y],x]$, where $x$, $y$, and
    $z$ are the Wirtinger generators corresponding to the dotted arcs
    in Figure~\ref{figure:borromean-conf}.  Also, $\eta$ is nontrivial
    in $\pi_1(X_L)/\pi_1(X_L)_3$.
  \end{enumerate}
\end{lemma}

Here $[a,b]$ denotes the commutator $aba^{-1}b^{-1}$.

\begin{proof}
  (1) Tubing the obvious disc bounded by $\eta$ along the components
  of $L$ that intersect it, we obtain a genus two surface $V$ with
  boundary $\eta$ which is shown in
  Figure~\ref{figure:borromean-satellite-grope}.  This is the body of
  the desired capped grope.  The whole capped grope is the body taken
  together with the four caps shown in
  Figure~\ref{figure:borromean-satellite-grope} as shaded discs.

  \begin{figure}[htb]
    \includegraphics{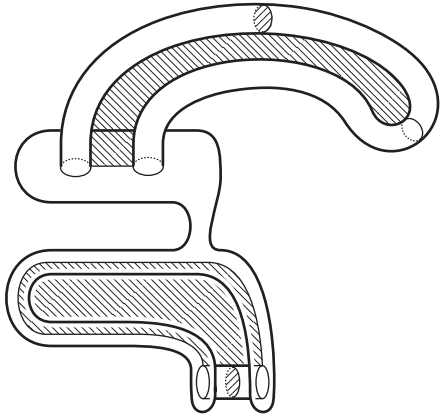}
    \caption{The capped grope bounded by~$\eta$.}
    \label{figure:borromean-satellite-grope}
  \end{figure}

  (2) The claim that $\eta=[x,y][[x,y],x]$ follows from a
  straightforward computation in terms of the Wirtinger generators,
  reading undercrossings of $\eta$ starting from the dot on $\eta$ in
  Figure~\ref{figure:borromean-conf}.  Since $L$ has vanishing linking
  number, due to Milnor~\cite{Milnor:1957-1} (see also
  Stallings~\cite{Stallings:1965-1}), $\pi_1(X_L)/\pi_1(X_L)_3$ is
  isomorphic to $F/F_3$ where $F$ is the free group generated by $x$,
  $y$, and~$z$.  Consequently, $[[x,y],x] \in \pi_1(X_L)_3$ and
  $[x,y]\notin \pi_1(X_L)_3$.  From this the second conclusion
  follows.
\end{proof}

As in
Section~\ref{subsection:satellite-configuration-of-iterated-Bing-doubles},
we apply the construction described just before
Lemma~\ref{lemma:grope-concordance-for-iterated-satellite-construction},
using our $(L,\eta)$ as the seed link $(L_0,\eta)$ and using the
Stevedore satellite configuration described in
\cite[Figure~6]{Cha:2012-1} (see our
Figure~\ref{figure:stevedore-conf}) as $(K_i,\alpha_i)$ for
$i=0,\ldots,n-2$ as above.  Let the resulting links be the~$L_j$.
Then by
Lemma~\ref{lemma:grope-concordance-for-iterated-satellite-construction},
the $L_j$ are height $n+2$ grope concordant to the Borromean rings
$L$, so these satisfy the first part of
Theorem~\ref{theorem:links-not-conc-hom-cobordant-0-surgeries}~(\ref{item:thm-links-not-conc-hom-cob-III-grope-concordance}).
The second part of
Theorem~\ref{theorem:links-not-conc-hom-cobordant-0-surgeries}~(\ref{item:thm-links-not-conc-hom-cob-III-grope-concordance}),
on the failure of the links to be pairwise height $n+3$ grope
concordant, will be shown in Section~\ref{section:amenable-signature}.

Furthermore, the links $L_j$ satisfies the hypothesis of
Theorem~\ref{Theorem:borromean_homology_equiv}.  First note that since
our satellite operation does not change the knot type of the
components, $L_j$ has unknotted components.  In particular the Arf
invariants of the components vanish.  Recall, from the proof of
Lemma~\ref{lemma:milnor-invariants-unchanged-satellite}, that there is
a homology equivalence $f\colon X_{L_j} \to X_{L_0}$ obtained from the
satellite construction $L_j = L_0(\eta, J_{n-1}^j)$; indeed $f$ is
obtained by glueing the identity map of $X_{L_0\sqcup \eta}$ with the
standard homology equivalence $(X_{J_{n-1}^j},\partial X_{J_{n-1}^j})
\to (S^1\times D^2, S^1\times S^1)$ along $S^1\times S^1$.  Since our
curve $\eta \subset S^3-L_0$ lies in the commutator subgroup of
$\pi_1(S^3-L_0)$, $f$ is indeed a $\Lambda$-homology equivalence
$X_{L_j} \to X_{L_0}$, by a Mayer-Vietoris argument.  Filling it in
with 3 solid tori, we obtain a $\Lambda$-homology equivalence $M_{L_j}
\to T^3=M_{L_0}$ as desired.  Therefore, by applying
Theorem~\ref{Theorem:borromean_homology_equiv}, it follows that the
links $L_j$ satisfy
Theorem~\ref{theorem:links-not-conc-hom-cobordant-0-surgeries}~(\ref{item:thm-links-not-conc-hom-cob-II-top-hom-cobordant}).

 We need to confirm that the $L_j$ satisfy
Theorem~\ref{theorem:links-not-conc-hom-cobordant-0-surgeries}~(\ref{item:thm-links-not-conc-hom-cob-I-non-diffeo}),
namely the $M_{L_j}$ are not homeomorphic.  The underlying idea is as
follows.  Recall that $L_j$ is defined by a satellite construction,
starting with a knot~$J^j_0$.  In many cases, the JSJ pieces of the
exterior of $J^j_0$ become parts of the JSJ decomposition of
$M_{L_j}$, so that the $M_{L_j}$ have distinct JSJ decompositions.
But a complete proof of this seems to require complicated arguments (a
technical issue is that an essential torus might not be parallel to a
JSJ torus, because of Seifert fibred pieces).  In order to avoid
these complications from JSJ decompositions, we will present a simpler
argument using only the number of incompressible tori; this is enough
for our purpose.

We need the following.  The Kneser-Haken finiteness
theorem~\cite{Haken:1961-1} states that for each 3-manifold $M$, there
is a bound, say $C_{KH}(M)$, on the number of disjoint pairwise
non-parallel incompressible surfaces that can be embedded in~$M$.
Recall that the knot $J_0^j$ used in the construction of the link
$L_j$ is a connected sum of $N_j$ knots, where $\{N_j\}$ was an
increasing sequence to be specified (see the paragraph before
Lemma~\ref{lemma:grope-concordance-for-iterated-satellite-construction}).
Here is the first requirement on the~$N_j$: we choose the $N_j$
inductively in such a way that $N_j>\max\{C_{KH}(M_{L_k})\mid
k=0,1,\ldots,j-1\}$.

\begin{lemma}
  \label{lemma:non-homeomorphic-zero-surgery}
  The zero surgery manifolds $M_{L_i}$ and $M_{L_j}$ are not
  homeomorphic for $i \neq j$.
\end{lemma}

\begin{proof}
  Recall that $M_{L_0} = M_L$ is the 3-torus~$T^3$.  Consider
  $Y:=M_L\sm \nu(\eta)$, where $\nu(\eta)$ is an open tubular
  neighbourhood of~$\eta$.  For notational convenience, denote the
  exterior of $J^j_{n-1}$ by $X:= X_{J^j_{n-1}}$.  The 3-manifold
  $M_{L_j}$ is obtained by glueing $Y$ and $X$ along their boundaries.
  Let $T=\partial Y=\partial X$ be the common boundary torus.  Note
  that $M_{L_0}$ can also be described in the same way, using
  $J^0_{n-1}:=$ unknot; in this case, the torus $T$ is compressible in
  $M_{L_0}$ since $X$ is a solid torus.

  \begin{claim}
    For $j\ge 1$, the torus $T$ is incompressible in~$Y$.
  \end{claim}

  Using the claim, we will show that the 3-manifolds $M_{L_j}$ are not
  pairwise homeomorphic.  Suppose $j\ge 1$.  Since the knot
  $J_{n-1}^j$ is obtained from an iterated satellite construction with
  the first stage knot $J_0^j$ a connected sum of $N_j$ nontrivial
  knots, the exterior $X$ of $J_{n-1}^j$ has at least $N_j$
  incompressible tori, including the boundary~$T$.  Since
  $M_{L_j}=Y\cup_T X$ and $T$ is incompressible in $Y$, it follows
  that there are $N_j$ non-parallel incompressible tori in~$M_{L_j}$.
  For any $k<j$, since $N_j > C_{KH}(M_{L_k})$, it follows that
  $M_{L_j}$ is not homeomorphic to~$M_{L_k}$.

  Now, to complete the proof, we will show the claim.  If there is an
  essential curve on $T$ which bounds a disc in $Y$, then it must be a
  zero-linking longitude, say $\eta'$, of $\eta$, since the meridian
  of $\eta$ is a generator of $H_1(Y\sm \eta)=\Z^4$.  By the following
  lemma, we have a contradiction.
\end{proof}

\begin{lemma}
  The class of $\eta'$ is nontrivial in the fundamental group
  $\pi_1(Y\sm\eta)$.
\end{lemma}

\begin{proof}
  We consider a Wirtinger presentation of $\pi_1(Y\sm\eta)$ given as
  follows: it has 24 generators denoted by $x_1,\ldots,x_{24}$
  associated to arcs in Figure~\ref{figure:borromean-conf}.  Here
  $(x_1,\ldots,x_{10})$, $(x_{11},x_{12})$, $(x_{13},\ldots,x_{16})$, and
  $(x_{17},\ldots,x_{24})$ are those associated to the arcs of the
  components $x$, $y$, $y$, and $\eta$, respectively.  In each
  component, the arc with a dot on it is the first one, and other arcs
  are ordered along the orientation.  There are 27 relators:
  \begin{gather*}
    x_{1} x_{11} \ol x_{1} \ol x_{12},\,
    x_{11} x_{1} \ol x_{11} \ol x_{2},\,
    x_{2} x_{18} \ol x_{2} \ol x_{19},\,
    x_{19} x_{2} \ol x_{19} \ol x_{3},\,
    x_{3} x_{20} \ol x_{3} \ol x_{21},\,
    \\
    x_{3} x_{23} \ol x_{3} \ol x_{22},\,
    x_{22} x_{4} \ol x_{22} \ol x_{3},\,
    x_{21} x_{4} \ol x_{21} \ol x_{5},\,
    x_{5} x_{16} \ol x_{5} \ol x_{13},\,
    x_{13} x_{5} \ol x_{13} \ol x_{6},\,
    \\
    x_{11} x_{7} \ol x_{11} \ol x_{6},\,
    x_{7} x_{11} \ol x_{7} \ol x_{12},\,
    x_{13} x_{8} \ol x_{13} \ol x_{7},\,
    x_{8} x_{16} \ol x_{8} \ol x_{15},\,
    x_{21} x_{9} \ol x_{21} \ol x_{8},\,
    \\
    x_{22} x_{9} \ol x_{22} \ol x_{10},\,
    x_{10} x_{23} \ol x_{10} \ol x_{24},\,
    x_{10} x_{20} \ol x_{10} \ol x_{19},\,
    x_{19} x_{1} \ol x_{19} \ol x_{10},\,
    x_{1} x_{18} \ol x_{1} \ol x_{17},\,
    \\
    x_{15} x_{22} \ol x_{15} \ol x_{21},\,
    x_{22} x_{15} \ol x_{22} \ol x_{14},\,
    x_{24} x_{13} \ol x_{24} \ol x_{14},\,
    x_{13} x_{24} \ol x_{13} \ol x_{17},\,
    \\
    \ol x_{11} \ol x_{19} x_{22} \ol x_{21} \ol x_{13} x_{11} x_{13} x_{21} \ol x_{22} x_{19},\,
    \ol x_{1} x_{7},\,
    \ol x_{24} x_{22} x_{8} \ol x_{5}.
  \end{gather*}
  Indeed, the first 24 are the standard Wirtinger relators for the
  4-component link $L \sqcup \eta$ (thus one of these may be omitted),
  and the last 3 relators arise from the zero surgery performed
  along~$L$.  It is straightforward to read off the curve $\eta'$:
  \[
  \eta' = x_{1} \ol x_{2} x_{10} \ol x_{3} x_{15} x_{3} \ol x_{10} \ol x_{13}.
  \]

  We define a representation $\rho\colon \pi_1(Y\sm\eta) \to
  \SL(2,\Z_5)$ by mapping the above 24 generators, respectively, to:
  \begin{gather*}
    \begin{bmatrix} 0 & 4 \\ 1 & 3 \end{bmatrix},\,
    \begin{bmatrix} 4 & 0 \\ 1 & 4 \end{bmatrix},\,
    \begin{bmatrix} 0 & 4 \\ 1 & 3 \end{bmatrix},\,
    \begin{bmatrix} 4 & 4 \\ 0 & 4 \end{bmatrix},\,
    \begin{bmatrix} 2 & 1 \\ 1 & 1 \end{bmatrix},\,
    \begin{bmatrix} 4 & 0 \\ 1 & 4 \end{bmatrix},\,
    \\
    \begin{bmatrix} 0 & 4 \\ 1 & 3 \end{bmatrix},\,
    \begin{bmatrix} 4 & 0 \\ 4 & 4 \end{bmatrix},\,
    \begin{bmatrix} 1 & 1 \\ 1 & 2 \end{bmatrix},\,
    \begin{bmatrix} 1 & 1 \\ 1 & 2 \end{bmatrix},\,
    \begin{bmatrix} 4 & 1 \\ 4 & 0 \end{bmatrix},\,
    \begin{bmatrix} 3 & 1 \\ 2 & 1 \end{bmatrix},\,
    \\
    \begin{bmatrix} 4 & 3 \\ 2 & 3 \end{bmatrix},\,
    \begin{bmatrix} 1 & 3 \\ 0 & 1 \end{bmatrix},\,
    \begin{bmatrix} 2 & 2 \\ 2 & 0 \end{bmatrix},\,
    \begin{bmatrix} 4 & 2 \\ 3 & 3 \end{bmatrix},\,
    \begin{bmatrix} 0 & 1 \\ 4 & 2 \end{bmatrix},\,
    \begin{bmatrix} 0 & 1 \\ 4 & 2 \end{bmatrix},\,
    \\
    \begin{bmatrix} 1 & 1 \\ 0 & 1 \end{bmatrix},\,
    \begin{bmatrix} 3 & 4 \\ 4 & 4 \end{bmatrix},\,
    \begin{bmatrix} 2 & 1 \\ 4 & 0 \end{bmatrix},\,
    \begin{bmatrix} 4 & 4 \\ 4 & 3 \end{bmatrix},\,
    \begin{bmatrix} 0 & 4 \\ 1 & 2 \end{bmatrix},\,
    \begin{bmatrix} 1 & 0 \\ 1 & 1 \end{bmatrix}.
  \end{gather*}
  It can be verified that all the relators are sent to the identity,
  by a straightforward computation.  Also, we have that
  \[
  \rho(\eta') = \begin{bmatrix} 3 & 1 \\ 4 & 0 \end{bmatrix}
  \]
  is not the identity.  (We found the representation $\rho$ using a
  computer program that performs a brute force search in the
  representation variety.)  This completes the proof.
\end{proof}

\section{Grope concordance and amenable signatures}
\label{section:amenable-signature}

In this section we show that the links described in
Sections~\ref{subsection:satellite-configuration-of-iterated-Bing-doubles}
and~\ref{subsection:non-diffeomorphic-zero-surgery} are not height
$n+3$ grope concordant, by using amenable signature obstructions from~\cite{Cha:2012-1}.
  In fact the amenable signatures we use are
obstructions to being $n$-solvably cobordant, which is a relative
analogue for manifolds with boundary, or bordered manifolds, of the
notion of $n$-solvability of \cite{Cochran-Orr-Teichner:1999-1}.  For
our purpose it suffices to consider the case of link exteriors; an
\emph{$n$-solvable cobordism} between the exteriors $X$ and $X'$ of
two links with the same number of components is a 4-manifold $W$ with
$\partial W=X\cup_{\partial}-X'$ satisfying the conditions described
in~\cite[Definition~2.8]{Cha:2012-1}, where the boundary tori of $X$
and $X'$ are identified along the zero framing.  Since we do not use
the defining condition right now, instead of spelling it out
here, we begin with its relationship to grope concordance.  The
following theorem originates from
\cite[Theorem~8.11]{Cochran-Orr-Teichner:1999-1}, and was given in our
context in \cite{Cha:2012-1}.

\begin{theorem}[\cite{Cha:2012-1} Theorems 2.16 and
  2.13, and Remark 2.11]
  \label{theorem:grope-concordance-implies-being-solvably-cobordant}
  If two links are height $n+2$ grope concordant then their exteriors
  are $n$-solvably cobordant as bordered 3-manifolds.
\end{theorem}

As our key ingredient to detect non-solvably-cobordant
3-manifolds and therefore non-grope-concordant links, we will use the Amenable
Signature Theorem, which was first introduced
in~\cite{Cha-Orr:2009-01} for homology cobordism of closed 3-manifolds
and then generalised to $n$-solvable cobordisms of bordered
3-manifolds in~\cite{Cha:2012-1}.  We state a special case which will
be sufficient for our purpose.  For a closed 3-manifold $M$ and a
homomorphism $\phi\colon \pi_1(M)\to G$, denote the von Neumann-Cheeger-Gromov $\rho$-invariant by $\rhot(M,\phi) \in
\R$.  See e.g.\
\cite[Section~5]{Cochran-Orr-Teichner:1999-1} as well as
\cite{Chang-Weinberger:2003-1,Harvey:2006-1,Cha:2006-1,Cha-Orr:2009-01}
for definitions and useful properties of $\rhot(M,\phi)$.  Precise
references for the properties that we need will be recalled as we go
along.

\begin{theorem}[A special case of
  {\cite[Amenable Signature Theorem~3.2]{Cha:2012-1}}]
  \label{theorem:amenable-signature-for-solvable-cobordism}
  Suppose $W$ is an $(n+1)$-solvable cobordism between two bordered
  3-manifolds $X$ and~$X'$, and $G$ admits a subnormal series
  \[
  G = G_0 \supset G_1 \supset \cdots \supset G_n \supset G_{n+1}=\{e\}
  \]
  with each quotient $G_i/G_{i+1}$ torsion-free abelian.
  Then for any $\phi\colon \pi_1(X\cup_\partial -X') \to G$ which
  factors through $\pi_1(W)$, we have $\rhot(X\cup_\partial
  -X',\phi)=0$.
\end{theorem}

Recall that in our construction of the links $L_j$ the knot $J_0^j$
was the connected sum of $N_j$ copies of Cochran-Teichner's knot,
say~$J$.  Now we proceed to specify the integers~$N_j$.  Denote, by
$\rhot(K):=\int_{S^1} \sigma_K(\omega) \,d\omega$, the integral of the
Levine-Tristram signature function over the circle normalized to
length~one.  For $M_K$ the zero surgery on $K$ and $\phi_0 \colon
\pi_1(M_K) \to \Z$ the abelianisation homomorphism, $\rhot(K)
=\rhot(M_K,\phi_0)$ by
\cite[Proposition~5.1]{Cochran-Orr-Teichner:2002-1}.  We have
$\rhot(J_0^j)=N_j\rhot(J) = 4N_j/3$ by additivity under connected sum
and \cite[Lemma~4.5]{Cochran-Teichner:2003-1}.  Due to Cheeger and
Gromov~\cite{Cheeger-Gromov:1985-1}, for any closed 3-manifold $Y$
there is a constant $C_Y > 0$ such that $|\rhot(Y,\psi)|< C_Y$ for
any~$\psi$.  From now on we abbreviate $\ell := n-k(m)$.  Define
\[
R:= C_{X_{L_0}\cup_\partial -X_{L_0}} + 2\sum_{i=0}^{\ell-1}
C_{M_{K_i}}.
\]
We choose the large integers $N_j$ inductively in such a way that
\[
N_j > 3R/4 + \max\{N_k\mid k<j\}.
\]
Then we have
\[
\rhot(J_0^j) > R + \rhot(J_0^k)
\]
whenever $j>k$.  For
Theorem~\ref{theorem:links-not-conc-hom-cobordant-0-surgeries}, we
make these choices so that the condition in the preamble to
Lemma~\ref{lemma:non-homeomorphic-zero-surgery} relating to the
Kneser-Haken bound is simultaneously satisfied.

Now we start the proof that our links $L_j$ are not height $n+3$ grope
concordant to one another.  Let $X$ and $X'$ be the exteriors of $L_j$
and $L_k$, respectively.  To distinguish them in the notation, we
denote the axis curve $\eta$ in $X$ by $\eta_j$, and we denote the
corresponding axis curve in $X'$ by $\eta_k$.

Recall that $m=|I|$ and that $k(m) = \lfloor \log_2(m-1) \rfloor$.  Also note that $k(m)+1 = \lceil \log_2(m) \rceil$.
By Theorem~\ref{theorem:grope-concordance-implies-being-solvably-cobordant}, it suffices to show the following:

\begin{proposition}
  \label{theorem:being-non-solvably-cobordant}
  For $n \geq k(m)$, the bordered 3-manifolds $X$ and $X'$ are not $(n+1)$-solvably
  cobordant when $j \ne k$.
\end{proposition}

By Theorem~\ref{theorem:grope-concordance-implies-being-solvably-cobordant},
it then follows that our links $L_j$ and $L_k$ are not height $n+3$ grope
concordant when $j \ne k$.

\begin{proof}
  The proof proceeds almost identically to that of
  \cite[Theorem~4.8]{Cha:2012-1}, which combines the Amenable
  Signature Theorem of \cite{Cha:2012-1} with a higher order
  Blanchfield duality argument for a certain 4-dimensional cobordism
  introduced in~\cite{Cochran-Harvey-Leidy:2009-1} (see our $W_0$
  below).  So we will give an outline for our case and discuss
  differences from \cite[Theorem~4.8]{Cha:2012-1}.

  Suppose $W$ is an $(n+1)$-solvable cobordism with $\partial W =
  X\cup_{\partial} -X'$.  Similarly to \cite[Section 4.3]{Cha:2012-1}
  (see the paragraph entitled ``Cobordism associated to an iterated
  satellite construction''), we consider a cobordism $V$ with
  \begin{multline*}
    \partial V = M_{J_0^j} \sqcup -M_{J_0^k} \sqcup M_{K_0} \sqcup
    -M'_{K_0} \sqcup \cdots \sqcup M_{K_{\ell-1}} \sqcup
    -M'_{K_{\ell-1}} \\
    \sqcup (X_{L_0}\cup_\partial -X_{L_0}) \sqcup -(X
    \cup_\partial-X')
  \end{multline*}
  which is built by stacking cobordisms associated to satellite
  constructions~\cite[p.~1429]{Cochran-Harvey-Leidy:2009-1}, where
  $M'_{K_i}$ is a copy of $M_{K_i}$, and then construct a cobordism
  $W_0$ with
  \[
  \partial W_0 = M_{J_0^j} \sqcup -M_{J_0^k} \sqcup M_{K_0} \sqcup -M'_{K_0}
  \sqcup \cdots \sqcup M_{K_{\ell-1}} \sqcup -M'_{K_{\ell-1}} \sqcup (X_{L_0}\cup_\partial
  -X_{L_0})
  \]
  by attaching $V$ to $W$ along $X\cup_\partial -X'$.  We omit the
  detailed construction of $V$ and $W_0$ but state a couple of useful
  facts which can be verified as in \cite[Section~4.3]{Cha:2012-1}.
  Let $\{\cP^r G\}$ be the rational derived series of a group~$G$,
  i.e.\ $\cP^0 G := G$ and $\cP^{r+1} G$ is the kernel of $\cP^r G \to
  H_1(\cP^r G;\Q)$.  Let $\phi_0$ be the quotient map $\pi_1(W_0) \to
  G:=\pi_1(W_0)/\cP^{n+1}\pi_1(W_0)$.  Also we denote by $\phi_0$ the
  restrictions of $\phi_0$ to the components of $\partial W_0$ and to
  $W\subset W_0$, as an abuse of notation.  Then we have the following
  facts.

  \begin{enumerate}
  \item \label{item-facts-1} $\displaystyle \rhot(M_{J_0^j},\phi_0) -
    \rhot(M_{J_0^k},\phi_0) + \rhot(X_{L_0}\cupover{\partial}
    -X_{L_0},\phi_0) \\ \hbox{}\hfill+ \sum_{i=0\mathstrut}^{\ell-1}
    \rhot(M_{K_i},\phi_0) - \sum_{i=0\mathstrut}^{\ell-1}
    \rhot(M'_{K_i},\phi_0) = \rhot(X\cupover{\partial}
    -X',\phi_0).$
   \item \label{item-facts-2} The image of the meridian of $J_0^j$ in
    $M_{J_0^j}\subset \partial W_0$ under $\phi_0$ is a nontrivial
    element in the torsion-free abelian subgroup
    $\cP^n\pi_1(W)/\cP^{n+1}\pi_1(W)$ of~$G$.  Similarly for $k$
    instead of~$j$.
  \end{enumerate}
  The proof of (\ref{item-facts-1}) is completely identical to that given in \cite[Section
  4.3]{Cha:2012-1} (see the paragraphs entitled ``Cobordism associated
  to an iterated satellite construction'' and ``applications of
  Amenable Signature Theorem''): briefly, the $\rhot$-invariant of
  $\partial W_0$, which is the left hand side of (\ref{item-facts-1}), is equal to the
  $L^2$-signature defect of $W_0=V\cup_{X\cup_\partial -X'}W$ (this is
  a standard fact from index theory, or can be taken as the definition
  of $\rhot$).  It turns out that $V$ has no contribution to the
  $L^2$-signature defect by
  \cite[Lemma~2.4]{Cochran-Harvey-Leidy:2009-1}.  So the left hand side of
  (\ref{item-facts-1}) is equal to the $L^2$-signature defect of $W$, which is the
  $\rhot$-invariant of $\partial W$, namely the right hand side of~(\ref{item-facts-1}).

  The proof of $($\ref{item-facts-2}$)$ is almost identical to that
  given in \cite[Theorem~4.10]{Cha:2012-1}.  Only the following change
  is required: in the initial step of the inductive argument in
  \cite[Theorem~4.10]{Cha:2012-1}, it was shown that the image of (a
  parallel copy of) $\eta \subset X \subset \partial W$ is nontrivial
  under the quotient map $\pi_1(W)\to \pi_1(W)/\cP^2 \pi_1(W)$ (see
  the fourth paragraph of \cite[Proof of Theorem~4.10]{Cha:2012-1}) using
  a Blanchfield duality argument.

  In our case, instead we use
  Lemma~\ref{lemma:dwyer-for-solvable-cobordism} below, which is a
  generalisation of \cite[Lemma~3.5]{Cha-Friedl-Powell:2012-1}, to
  show that the image of $\eta$ is nontrivial in
  $\pi_1(W)/\cP^{(k(m)+1)} \pi_1(W)$.  The argument used in
  Lemma~\ref{lemma:dwyer-for-solvable-cobordism} is essentially an
  application Dwyer's theorem.

  \begin{lemma}
    \label{lemma:dwyer-for-solvable-cobordism}
    If $W$ is an $n$-solvable cobordism between two link exteriors (or
    more generally bordered 3-manifolds) $X$ and $X'$, then the
    inclusions induce isomorphisms
  \[
  \pi_1(X)/\pi_1(X)_{q}
  \cong \pi_1(W)/\pi_1(W)_{q} \cong \pi_1(X')/\pi_1(X')_{q}.
  \]
  for $q \leq 2^n+1$.
  \end{lemma}

  \begin{proof}
    Recall Dwyer's theorem~\cite{Dwyer:1975-1}: if $f\colon X\to Y$
    induces an isomorphism $H_1(X;\Z)\cong H_1(Y;\Z)$ and an epimorphism
    \[
    H_2(X;\Z)\to H_2(Y;\Z)/\Im \{ H_2(Y;\Z[\pi_1(W)/\pi_1(W)_q]) \rightarrow
    H_2(Y;\Z)\},
    \]
    then $f$ induces an isomorphism $\pi_1(X)_q/\pi_1(X)_{q+1} \cong
    \pi_1(Y)_q/\pi_1(Y)_{q+1}$.

    In our case, by the definition of an $n$-solvable cobordism
    \cite[Definition~2.8]{Cha:2012-1}, we have $H_1(X;\Z)\cong
    H_1(W;\Z)\cong H_1(X';\Z)$, where the isomorphisms are induced by
    the inclusion maps.  Also, by \cite[Definition~2.8]{Cha:2012-1}
    there are elements $\ell_1,\ldots,\ell_r$, $d_1,\ldots,d_r$ lying
    in $H_2(W;\Z[\pi_1(W)/\pi_1(W)^{(n)}])$ such that the images of
    $\ell_i$ and $d_j$ generate~$H_2(W;\Z)$.  Since $\pi_1(W)^{(n)}$
    is contained in $\pi_1(W)_{2^{n}}$, the $H_2$ condition of Dwyer's
    theorem is satisfied.  Therefore it follows that
    \[
    \pi_1(X)_q/\pi_1(X)_{q+1} \cong
    \pi_1(W)_q/\pi_1(W)_{q+1} \cong
    \pi_1(X')_q/\pi_1(X')_{q+1}
    \]
    for $q \leq 2^n$ by Dwyer's theorem.
    From this the desired conclusion follows by the five
    lemma.
  \end{proof}

  Recall that $\eta\subset X$ represents a nontrivial element in
  $\pi_1(X)/\pi_1(X)_{m}$ by
  Lemma~\ref{lemma:properties-of-L_I-and-eta}~(\ref{item:properties-of-L_I-eta-two}).
  Since the above isomorphisms preserve longitudes (and meridians),
  $\eta_j \subset X$ represents a nontrivial element in
  $\pi_1(W)/\pi_1(W)_{m}$.  Since $L_j$ has vanishing Milnor
  invariants of length less than $|I|=m$, we have
  $\pi_1(X)/\pi_1(X)_{m} \cong F/F_{m}$, where $F$ is the free group
  with rank $m$, by \cite[Theorem~4]{Milnor:1957-1}.  Consequently
  $\pi_1(W)/\pi_1(W)_{m}$ is torsion free.

  We note that for any group $\pi$, we have $\pi^{(k(q)+1)} =
  \pi^{(\lceil \log_2(q) \rceil)} \subseteq \pi_q$.  Therefore there
  is a quotient map $\pi_1(W)/\pi_1(W)^{(k(m)+1)} \to
  \pi_1(W)/\pi_1(W)_{m}$, and this map factors through
  $\pi_1(W)/\cP^{(k(m)+1)}\pi_1(W)$ by the definition
  of~$\cP^{(k(m)+1)}$ and the fact that the codomain is torsion free.
  Since $\eta_j$ is nontrivial in $\pi_1(W)/\pi_1(W)_{m}$, $\eta_j$ is
  also nontrivial in $\pi_1(W) / \cP^{(k(m)+1)}\pi_1(W)$.  By
  replacing $j$ with $k$ and $X$ with $X'$ we obtain the corresponding
  fact for $\eta_k$ in $X'$.

  To complete the proof of
  Proposition~\ref{theorem:being-non-solvably-cobordant}, we proceed as in
  \cite[Section 4.3]{Cha:2012-1}.  Observe that for the normal
  subgroups $G_i := \cP^i \pi_1(W_0)/\cP^{n+1}\pi_1(W_0)$ of our $G$,
  the quotient $G_i/G_{i+1}$ is torsion-free abelian.  So by Amenable
  Signature
  Theorem~\ref{theorem:amenable-signature-for-solvable-cobordism} we
  have $\rhot(X\cup_\partial -X')=0$.  Since the curve $\eta_j$
  represents a nontrivial element in a torsion-free abelian normal
  subgroup of $G$, the image of
  $\pi_1(M_{J_0^j})$ in $G$ under $\phi_0$ is the infinite cyclic
  group.  By $L^2$-induction (see e.g.\
  \cite[page~8~(2.3)]{Cheeger-Gromov:1985-1},
  \cite[Proposition~5.13]{Cochran-Orr-Teichner:1999-1}) and
  \cite[Proposition~5.1]{Cochran-Orr-Teichner:2002-1}, we have
  $\rhot(M_{J_0^j},\phi_0) = \rhot(J_0^j)$, and similarly for~$J_0^k$.
  Now combining these two facts with $($\ref{item-facts-1}$)$ we
  obtain:
  \begin{multline*}
    \rhot(J_0^j) - \rhot(J_0^k) + \rhot(X_{L_0}\cupover{\partial}
    -X_{L_0},\phi_0)
    \\ + \sum_{i=0\mathstrut}^{\ell-1}
    \rhot(M_{K_i},\phi_0) - \sum_{i=0\mathstrut}^{\ell-1} \rhot(M'_{K_i},\phi_0) = 0.
    \tag{$3$}
  \end{multline*}
  Recall that
  \begin{multline*}
    \bigg|\rhot(X_{L_0}\cupover{\partial} -X_{L_0},\phi_0) +
    \sum_{i=0\mathstrut}^{\ell-1} \rhot(M_{K_i},\phi_0) -
    \sum_{i=0\mathstrut}^{\ell-1} \rhot(M'_{K_i},\phi_0) \bigg|
    \\[-2ex]
    < R:= C_{X_{L_0}\cup_\partial -X_{L_0}} + 2\sum_{i=0}^{\ell-1}
    C_{M_{K_i}},
  \end{multline*}
  and in the preamble to
  Proposition~\ref{theorem:being-non-solvably-cobordant}, we chose
  $N_j$ in such a way that
  \[
  \big| \rhot(J_0^k) - \rhot(J_0^j) \big| > R
  \]
  whenever $k \ne j$.  Therefore $(3)$ implies that $j=k$.  Thus the
  existence of the $(n+1)$-solvable cobordism $W$ implies that $j=k$,
  which is the contrapositive of the desired statement.
\end{proof}

\def\MR#1{}
\bibliographystyle{alphaabbrv}
\bibliography{research}

\end{document}